\documentclass[a4paper]{amsart}

\usepackage{txfonts, amsmath,amstext,amsthm,amscd,amsopn,verbatim,amssymb, amsfonts}
\usepackage{fullpage}
\usepackage{mathtools}

\usepackage{float}
\restylefloat{table}

\usepackage{tikz}
\usepackage{tikz-cd}

\usetikzlibrary{matrix}
\usetikzlibrary{shapes}
\usetikzlibrary{arrows}
\usetikzlibrary{calc,3d}
\usetikzlibrary{decorations,decorations.pathmorphing,decorations.pathreplacing,decorations.markings}
\usetikzlibrary{through}

\tikzset{ext/.style={circle, draw,inner sep=1pt},int/.style={circle,draw,fill,inner sep=1.4pt},nil/.style={inner sep=1pt},dot/.style={circle,fill,inner sep=0pt,outer sep=0pt}}
\tikzset{cy/.style={circle,draw,fill,inner sep=2pt},scy/.style={circle,draw,inner sep=2pt},scyx/.style={draw,cross out,inner sep=2pt},scyt/.style={draw,regular polygon,regular polygon sides=3,inner sep=0.95pt}}
\tikzset{exte/.style={circle, draw,inner sep=3pt},inte/.style={circle,draw,fill,inner sep=3pt}}
\tikzset{diagram/.style={matrix of math nodes, row sep=3em, column sep=2.5em, text height=1.5ex, text depth=0.25ex}}
\tikzset{diagram2/.style={matrix of math nodes, row sep=0.5em, column sep=0.5em, text height=1.5ex, text depth=0.25ex}}

\tikzset{
  rightblue/.style={
    decoration={markings,mark=at position .8 with {\arrow[scale=1.2,blue]{latex}}},
    postaction={decorate},begin
    shorten >=0.4pt}}
\tikzset{
  leftblue/.style={
    decoration={markings,mark=at position .6 with {\arrowreversed[scale=1.2,blue]{latex}}},
    postaction={decorate},
    shorten >=0.4pt}}
\tikzset{
  rightred/.style={
    decoration={markings,mark=at position .4 with {\arrow[scale=1.2,red]{latex}}},
    postaction={decorate},
    shorten >=0.4pt}}
\tikzset{
  leftred/.style={
    decoration={markings,mark=at position .2 with {\arrowreversed[scale=1.2,red]{latex}}},
    postaction={decorate},
    shorten >=0.4pt}}

\tikzset{
  rightblue/.style={
    decoration={markings,mark=at position .8 with {\arrow[scale=1.2,blue]{latex}}},
    postaction={decorate},
    shorten >=0.4pt}}
\tikzset{
  leftblue/.style={
    decoration={markings,mark=at position .6 with {\arrowreversed[scale=1.2,blue]{latex}}},
    postaction={decorate},
    shorten >=0.4pt}}
\tikzset{
  rightred/.style={
    decoration={markings,mark=at position .4 with {\arrow[scale=1.2,red]{latex}}},
    postaction={decorate},
    shorten >=0.4pt}}
\tikzset{
  leftred/.style={
    decoration={markings,mark=at position .2 with {\arrowreversed[scale=1.2,red]{latex}}},
    postaction={decorate},
    shorten >=0.4pt}}

\tikzset{
  crossed/.style={
    decoration={markings,mark=at position .5 with {\arrow{|}}},
    postaction={decorate},
    shorten >=0.4pt}}

\tikzset{snakeit/.style={decorate, decoration={snake, amplitude=.2mm,segment length=1mm}}}

\newcommand{\Ed}{{
\begin{tikzpicture}[baseline=-.8ex,scale=.5]
\node[nil] (a) at (0,0) {};
\node[nil] (b) at (1,0) {};
\draw (a) edge[-latex] (b);
\end{tikzpicture}}}

\newcommand{\dE}{{
\begin{tikzpicture}[baseline=-.8ex,scale=.5]
\node[nil] (a) at (0,0) {};
\node[nil] (b) at (1,0) {};
\draw (a) edge[latex-] (b);
\end{tikzpicture}}}

\newcommand{\EdE}{{
\begin{tikzpicture}[baseline=-.65ex,scale=.5]
 \node[nil] (a) at (0,0) {};
 \node[int] (b) at (1,0) {};
 \node[nil] (c) at (2,0) {};
 \draw (a) edge[-latex] (b);
 \draw (b) edge[latex-] (c);
\end{tikzpicture}}}

\newcommand{\dEd}{{
\begin{tikzpicture}[baseline=-.65ex,scale=.5]
 \node[nil] (a) at (0,0) {};
 \node[int] (b) at (1,0) {};
 \node[nil] (c) at (2,0) {};
 \draw (a) edge[latex-] (b);
 \draw (b) edge[-latex] (c);
\end{tikzpicture}}}

\newcommand{\Ess}{{
\begin{tikzpicture}[baseline=-.65ex,scale=.5]
 \node[nil] (a) at (0,0) {};
 \node[nil] (c) at (1.4,0) {};
 \draw (a) edge[crossed,->] (c);
\end{tikzpicture}}}

\newcommand{\EE}{{
\begin{tikzpicture}[baseline=-.65ex,scale=.5]
 \node[nil] (a) at (0,0) {};
 \node[nil] (c) at (1,0) {};
 \draw (a) edge[very thick,snakeit] (c);
\end{tikzpicture}}}

\newcommand{\EEpassing}{{
\begin{tikzpicture}[baseline=-.65ex,scale=.5]
 \node[nil] (a) at (0,0) {};
 \node[int] (b) at (1,0) {};
 \node[nil] (c) at (2,0) {};
 \draw (a) edge[-latex] (b);
 \draw (b) edge[-latex] (c);
\end{tikzpicture}}}

\usepackage{chngcntr}
\counterwithin{figure}{section}
\theoremstyle{plain}
  \newtheorem{thm}[figure]{Theorem}
  \newtheorem*{thm*}{Theorem}
  
  \newtheorem{prop}[figure]{Proposition}
  
  \newtheorem{cor}[figure]{Corollary}
  
  \newtheorem{lemma}[figure]{Lemma}
\theoremstyle{definition}

\newcommand{\R}{{\mathbb{R}}}

\newcommand{\K}{{\mathbb{K}}}
\newcommand{\Z}{{\mathbb{Z}}}

\newcommand{\sym}{{\mathbb{S}}}

\newcommand{\grt}{\mathfrak{grt}}

\newcommand{\GC}{\mathsf{GC}}

\newcommand{\OGC}{\mathsf{OGC}}

\newcommand{\G}{\mathsf{G}}

\newcommand{\OG}{\mathsf{OG}}

\newcommand{\mV}{\mathsf{V}}
\newcommand{\mE}{\mathsf{E}}

\newcommand{\mO}{\mathsf{O}}

\DeclareMathOperator{\sgn}{sgn}
\DeclareMathOperator{\id}{id}
\DeclareMathOperator{\Def}{Def}

\newcommand{\grac}{\mathsf{grac}}

\newcommand{\bu}{{\bullet}}

\usepackage{graphicx}
\usepackage[all]{xy}
\usepackage{latexsym}
\usepackage{amssymb}
\usepackage{amsmath}
\usepackage{color}

\xyoption{arc}

\def\id{{\mbox{1 \hskip -7pt 1}}}
 \newcommand{\lon}{\longrightarrow}
 \newcommand{\lonm}{\longmapsto}

 \newcommand{\rar}{\rightarrow}

\newcommand{\p}{{\partial}}
\newcommand{\Id}{{\mathrm{Id}}}

\newcommand{\Der}{\mathrm{Der}}
 \newcommand{\bS}{{\mathbb S}}

 \newcommand{\ot}{\otimes}

\newcommand{\st}{\stackrel{\leftrightarrow}}
\newcommand{\dGC}{\mathsf{dGC}}

 \newcommand{\Beq}{\begin{equation}}
 \newcommand{\Eeq}{\end{equation}}
 \newcommand{\Beqr}{\begin{eqnarray}}
 \newcommand{\Eeqr}{\end{eqnarray}}
 \newcommand{\Beqrn}{\begin{eqnarray*}}
 \newcommand{\Eeqrn}{\end{eqnarray*}}
 \newcommand{\Ba}{\begin{array}}
 \newcommand{\Ea}{\end{array}}
 \newcommand{\Bi}{\begin{itemize}}
 \newcommand{\Ei}{\end{itemize}}
 \newcommand{\Bc}{\begin{center}}
 \newcommand{\Ec}{\end{center}}

 \newcommand{\cA}{{\mathcal A}}
 \newcommand{\cB}{{\mathcal B}}
 
 \newcommand{\caD}{{\mathcal D}}
 \newcommand{\cE}{{\mathcal E}}
 \newcommand{\cF}{{\mathcal F}}

 \newcommand{\cP}{{\mathcal P}}
 \newcommand{\cQ}{{\mathcal Q}}
 \newcommand{\cR}{{\mathcal R}}

 \newcommand{\cU}{{\mathcal U}}

 \newcommand{\Ga}{\Gamma}

 \newcommand{\sip}{\smallskip}
 \newcommand{\bip}{\bigskip}
 
\newcommand{\LB}{\mathcal{L}\mathit{ieb}}

\newcommand{\LBd}{\mathcal{L}\mathit{ieb}_{d}}

\newcommand{\HoLBd}{\mathcal{H}\mathit{olieb}_{d}}

\newcommand{\wHoLBd}{\widehat{\mathcal{H}\mathit{olieb}}_{d}}

\newcommand{\HoLB}{\mathcal{H}\mathit{olieb}}

\newcommand{\Assb}{\mathcal{A}\mathit{ssb}}
\begin{document}
\title{Quantizations of Lie bialgebras, duality involution and\\
oriented graph complexes}

\author{Sergei Merkulov and Marko \v Zivkovi\' c}
\address{Department Mathematics, University of Luxembourg\\
Maison du Nombre, 6, avenue de la Fonte, L-4364 Esch-sur-Alzette\\
Grand Duchy of Luxembourg}


\keywords{Lie bialgebras, deformation quantization, graph complexes}

\begin{abstract}  We prove that the action of the Grothendieck-Teichm\"uller group on the genus completed properad of (homotopy) Lie bialgebras commutes with the reversing directions involution  of the latter. We  also prove that  every universal quantization of Lie bialgebras is homotopy equivalent to the one which commutes with the duality involution exchanging Lie bracket and Lie cobracket. The proofs are based on a new result in the theory of oriented graph complexes (which can be of independent interest) saying that the involution on an oriented graph complex that changes all directions on edges induces the identity map on its cohomology.
\end{abstract}

\maketitle

\section{Introduction} Quantum groups originated in the theory of quantum integrable systems, and  found many application in pure mathematics and mathematical physics. They have been identified by Vladimir Drinfeld and Michio Jimbo with a certain class of Hopf algebras. One of the most famous examples of Hopf algebras is given by the universal enveloping  algebra $\cU(V)$ associated to a Lie algebra $V$; it can be understood as a deformation of the standard commutative and co-commutative Hopf algebra structure in the symmetric tensor algebra $\odot^\bu V$ which is controlled by the given Lie algebra structure in $V$. Generic deformations of that standard Hopf algebra structure in $\odot^\bu V$ are controlled by a more informative datum, by a so called {\em Lie bialgebra}\, structure in $V$  which is based on two compatible Lie and co-Lie operations \cite{Dr1},
$$
\Ba{c}\resizebox{6mm}{!}{  \xy
(0,-7)*{_{1\hspace{5mm} 2}},
(0,5)*{}="U1";
(0,0)*{\bu}="o";
(-3,-5)*{}="D1";
(3,-5)*{}="D2";
\ar @{->} "D1";"o" <0pt>
\ar @{->} "D2";"o" <0pt>
\ar @{<-} "U1";"o" <0pt>
 \endxy}
 \Ea  \Leftrightarrow [\ ,\ ]: \Lambda^2 V \rar V, \ \ \ \ \ \
 \Ba{c}\resizebox{6mm}{!}{  \xy
(0,7)*{_{1\hspace{5mm} 2}},
(0,-5)*{}="U1";
(0,0)*{\bu}="o";
(-3,5)*{}="D1";
(3,5)*{}="D2";
\ar @{<-} "D1";"o" <0pt>
\ar @{<-} "D2";"o" <0pt>
\ar @{->} "U1";"o" <0pt>
 \endxy}
 \Ea \Leftrightarrow \triangle: V \rar \Lambda^2 V
$$
which in terms of properads get represented by directed graphs with two (resp. one) inputs and one (resp.\ two) outputs. If $V$ is a Lie algebra, then its dual\footnote{We work in the category of possibly infinite-dimensional vector spaces $V$ which are direct limits, $\displaystyle V=\lim_{\lon} V_p$, of finite-dimensional ones. Their duals are defined as projective limits, $\displaystyle V^*=\lim_{\longleftarrow} V^*_p$.} vector space $V^*$ is also a Lie bialgebra. This duality transformation looks most nicely when expressed in the language of prop(erad)s --- all the relations defining the notion of the properad $\LB$ controlling Lie bialgebras  are invariant under the involution which reverses directions on all edges simultaneously,
\Beq\label{1: involution of Lie and coLie}
\imath: \Ba{c}\resizebox{6mm}{!}{  \xy
(0,-7)*{_{1\hspace{5mm} 2}},
(0,5)*{}="U1";
(0,0)*{\bu}="o";
(-3,-5)*{}="D1";
(3,-5)*{}="D2";
\ar @{->} "D1";"o" <0pt>
\ar @{->} "D2";"o" <0pt>
\ar @{<-} "U1";"o" <0pt>
 \endxy}
 \Ea \lon \Ba{c}\resizebox{6mm}{!}{  \xy
(0,7)*{_{1\hspace{5mm} 2}},
(0,-5)*{}="U1";
(0,0)*{\bu}="o";
(-3,5)*{}="D1";
(3,5)*{}="D2";
\ar @{<-} "D1";"o" <0pt>
\ar @{<-} "D2";"o" <0pt>
\ar @{->} "U1";"o" <0pt>
 \endxy}
 \Ea, \ \ \ \
 \Ba{c}\resizebox{6mm}{!}{  \xy
(0,7)*{_{1\hspace{5mm} 2}},
(0,-5)*{}="U1";
(0,0)*{\bu}="o";
(-3,5)*{}="D1";
(3,5)*{}="D2";
\ar @{<-} "D1";"o" <0pt>
\ar @{<-} "D2";"o" <0pt>
\ar @{->} "U1";"o" <0pt>
 \endxy}
 \Ea
 \lon
 \Ba{c}\resizebox{6mm}{!}{  \xy
(0,-7)*{_{1\hspace{5mm} 2}},
(0,5)*{}="U1";
(0,0)*{\bu}="o";
(-3,-5)*{}="D1";
(3,-5)*{}="D2";
\ar @{->} "D1";"o" <0pt>
\ar @{->} "D2";"o" <0pt>
\ar @{<-} "U1";"o" <0pt>
 \endxy}
 \Ea
 \Eeq
We address in this paper two questions concerning this involution.

\sip

{\sf Question 1}: It is proven in \cite{MW1} that the mysterious
Grothendieck-Teichm\"uller group $GRT$ acts faithfully (and essentially transitively) on the  genus completion of the properad $\LB$ of Lie bialgebras. Does this action commute with the above direction (or duality) involution $\imath$? We prove that the answer to this question is {\sf yes}: it is enough to know the action of an element of  $GRT$ on just Lie (or coLie) generator of $\LB$, the value on the second one (up to homotopy equivalence) follows from the first one by reversing all arrows (see \S 2 and \S 4  for a precise formulation which takes care about orientations and signs). The same answer holds true when one considers the action of $GRT$ on the genus completion of the minimal resolution $\HoLB$ of the properad $\LB$.

\sip

We conjecture, however, that the above answer to Question 1 is negative if one replaces $\HoLB$ (or $\LB$) with its wheeled closure, $\HoLB^\circlearrowright$ (or $\LB^\circlearrowright$), which controls homotopy Lie bialgebra structures solely in {\em finite}-dimensional spaces (cf.\ \cite{Me}).

\sip

{\sf Question 2}: According to the famous theorem of Etingof-Kazhdan \cite{EK}, any Lie bialgebra structure in a vector space $V$ can be deformation quantized, that is, there is a continuous  Hopf algebra structure in  $\odot^\bu V[[\hbar]]$, $\hbar$ being a formal parameter, which --- at the infinitesimal in $\hbar$ level --- induces the given Lie bialgebra in $V$.
Moreover, there is a deformation quantization procedure which is universal in the sense that it is applicable to any (possibly, infinite-dimensional) Lie bialgebra $V$. Such a universal quantization is best understood as a morphism props \cite{MW2},
$$
\cQ: \Assb \lon \caD(\widehat{\LB})
$$
where $\Assb$ is the prop of associative bialgebras, and $\caD$ is a polydifferential endofunctor \cite{MW2} in the category of dg props applied to the genus completed prop(erad) $\widehat{\LB}$
of Lie bialgebras. It was proven in \cite{MW2} that such morphisms (up to homotopy, or gauge equivalence) are classified by the set of V.\ Drinfeld associators. Both sides of the above arrow admit a canonical direction involution as the relations defining associative bialgebras are also invariant under the direction (or duality) involution, in a full analogy to the case of Lie bialgebras. Are there universal quantizations of Lie bialgebras $\cQ$ which commute with the direction involutions of both sides, of $\cA ss\cB$ and  $\LB$? It was proven in \cite{KT}  that the answer is {\sf yes}: there exists at least one such {\em direction involution invariant}\, universal quantization. We classify (up to homotopy equivalence) in this paper all the direction involution invariant universal quantizations and prove that the set of all such equivalence classes  can identified with the set of V.\ Drinfeld associators. This result implies that {\em any universal quantization of (possibly, infinite-dimensional) Lie bialgebras is homotopy equivalent to a direction involution invariant one}. Thus, up to a homotopy equivalence, one can always make the formula for the deformed associative co-multiplication in $\odot^\bu V[\hbar]]$  to be simply the dual of the formula for the deformed multiplication in $\odot^\bu V^*[[\hbar]]$ corresponding to the quantization of the dual  Lie bialgebra (one must be careful about orientations and signs
as explained below). We conjecture that this statement is not true when applied to universal quantizations of solely {\em finite}-dimensional Lie bialgebras, i.e.\ to the universal quantizations of the form
$$
\cQ^\circlearrowright: \Assb^\circlearrowright \lon \caD(\widehat{\LB^{\circlearrowright}}).
$$
where the superscipt $^\circlearrowright$ stands for the wheeled closure of ordinary props \cite{Me,MMS}.

\sip

The above results are established with the help of a new fundamental fact about {\it oriented}\, graph complexes which is of independent interest. M.\ Kontsevich graph complexes\footnote{Graph complexes are often assumed to be generated by at least trivalent graphs. The superscript $\geq 2$ in the notation $\GC_d^{\geq 2}$ means that allow graphs with at least bivalent vertices. The difference between these two versions of M.Konstevich's graph complexes is well understood --- it is measured by a countable family of polytope-like graphs whose every vertex is precisely bivalent \cite{Wi1}.} $\GC_d^{\geq 2}$ come in a family parameterized by an integer $d\in \Z$ \cite{Kont3,Wi1}; complexes with the same parity of $d$ are isomorphic to each other (up to degree shift). The dg Lie algebra $\GC_2^{\geq 2}$ is of special interest as its zero-th cohomology group can be identified \cite{Wi1} with the Lie algebra,
$$
H^0(\GC_2^{\geq 2})=\grt_1,
$$
 of the famous prounipotent Grothendieck-Teichm\"uller group $GRT_1$ introduced in \cite{Dr2}. The graphs from $\GC_d^{\geq 2}$ have no directions fixed on its edges; the version $\dGC_d^{\geq 2}$ of $\GC_d^{\geq 2}$ in which that directions are fixed gives essentially nothing new as the natural map
 $$
 \GC_d^{\geq 2} \lon \dGC_d^{\geq 2}
 $$
which sends an undirected graph from the l.h.s.\ into a sum of directed graphs by choosing a direction on each edge in both possible ways, is a quasi-isomorphism \cite{Wi1}. Remarkably, the complex
$\dGC_d^{\geq 2}$ contains a subcomplex $\OGC_d$ spanned by graphs with {\it no closed paths of directed edges}, and there is an isomorphism of cohomology groups,
$$
H^\bu(\GC_d^{\geq 2})=H^\bu(\OGC_{d+1}), \ \ \  \forall\ d\in \Z,
$$
so that the Grothendieck-Teichm\"uller Lie algebra gets a new incarnation in dimension $d=3$,
$$
H^0(\OGC_3)=\grt_1.
$$
The oriented graph complexe $\OGC_d$ admits an involution $\imath: \OGC_d \rar \OGC_d$ which reverses directions on all edges  and hence decomposes into a direct sum
$$
\OGC_d=\OGC_d^+ \oplus \OGC_d^-,
$$
corresponding to eigenvalues $+1$ and $-1$ of the automorphism $\imath$. It is proven in this paper that {\it the subcomplex $\OGC_d^-$ is acyclic}\, implying, in particular, that
every element of the Grothendieck-Teichm\"uller Lie algebra $\grt_1$ can be represented by a linear combination of graphs which are all {\it invariant}\, under the reversing direction involution,
$$
H^0(\OGC_3^+)=\grt_1.
$$
This result is the origin of the above claims about  about universal deformation quantizations and automorphism groups of Lie bialgebras.

\subsection{Structure of the paper} In section \S 2 we remind the notion of (degree $d$ shifted) Lie bialgebras and of the properads, $\LB_d$ and $\HoLBd$, controlling their homotopy theory; we define an involution $\imath$ on $\HoLBd$ (and on its cohomology $\LBd$) which reverses directions on all edges. In Section \S 3 we recall the definition of oriented graph complex $\mO\G_{d}$, the dual of the above mentioned graph complex $\OGC_d$, and its skeleton version; we also introduce a direction reversing involution $\iota$ of $\mO\G_d$ and prove that its
eigenvalue $-1$ subcomplex is acyclic which implies the above claim about acyclicity of $\OGC_d^-$. In \S 3 we prove the answer {\sf yes} to the {\sf Question 1}, and in \S 4 we classify all direction involution invariant universal quantizations of Lie bialgebras, and prove that any universal quantization is homotopy equivalent to a direction involution invariant one.

\subsection{Notation}
 The set $\{1,2, \ldots, n\}$ is abbreviated to $[n]$;  its group of automorphisms is
denoted by $\bS_n$;
the trivial one-dimensional representation of
 $\bS_n$ is denoted by $\id_n$, while its one dimensional sign representation is
 denoted by $\sgn_n$.
The cardinality of a finite set $I$ is denoted by $\# I$.
We work throughout in the category of differential $\Z$-graded (dg, for short) vector spaces over a field $\K$
of characteristic zero; all our differentials have degree +1.
If $V=\oplus_{i\in \Z} V^i$ is a graded vector space, then
$V[k]$ stands for the graded vector space with $V[k]^i:=V^{i+k}$; for $v\in V^i$ we set $|v|:=i$. For a finite set $S$ we denote by $\K S$ its linear span over a field $\K$.

\bip

{\Large
\section{\bf Direction involution on properad of Lie bialgebras}
}

\subsection{Properad of Lie bialgebras} A degree $d\in \Z$ shifted Lie bialgebra structure in a graded vector space   is, by definition, a pair of linear maps
$$
[\ ,\ ]: \left\{ \Ba{cc} \Lambda^2 V \rar V[1-d] & \text{if $d$ is odd} \\
                          \odot^2 V \rar V[1-d] & \text{if $d$ is even} \Ea  \right. ,  \  \ \ \
\Delta: \left\{ \Ba{cc}  V \rar \Lambda^2 V[1-d] & \text{if $d$ is odd} \\
                          V \rar \odot^2 V[1-d]  & \text{if $d$ is even} \Ea \right.
$$
such that $[\ ,\ ]$ satisfies the Jacobi identity, $\Delta$ satisfies the co-Jacobi identity, and both operations satisfy the Drinfeld compatibility condition (which is best understood in terms of properads, see below). The case $d=1$ corresponds to the classical notion of a Lie bialgebra structure in $V$ which was introduced by Vladimir Drinfeld in the context of the deformation theory of the universal enveloping algebras in the category of quantum groups \cite{Dr1}.
The properad $\LBd$ governing degree $d$ Lie bialgebras is the quotient
$$
\LB_{d}:=\cF ree\langle E_0\rangle/\langle\cR\rangle,
$$
of the free prop(erad) generated by an  $\bS$-bimodule $E_0=\{E_0(m,n)\}_{m,n\geq 0}$ with
 all $E_0(m,n)=0$ except
  $$
E_0(2,1):=\id_1\ot (\sgn_2)^{\ot |d|}[d-1]=\mbox{span}\left\langle
\Ba{c}\resizebox{6mm}{!}{  \xy
(0,7)*{_{1\hspace{5mm} 2}},
(0,-5)*{}="U1";
(0,0)*{\bu}="o";
(-3,5)*{}="D1";
(3,5)*{}="D2";
\ar @{<-} "D1";"o" <0pt>
\ar @{<-} "D2";"o" <0pt>
\ar @{->} "U1";"o" <0pt>
 \endxy}
 \Ea
=(-1)^{d}
\Ba{c}\resizebox{6mm}{!}{  \xy
(0,7)*{_{2\hspace{5mm} 1}},
(0,-5)*{}="U1";
(0,0)*{\bu}="o";
(-3,5)*{}="D1";
(3,5)*{}="D2";
\ar @{<-} "D1";"o" <0pt>
\ar @{<-} "D2";"o" <0pt>
\ar @{->} "U1";"o" <0pt>
 \endxy}
 \Ea
   \right\rangle
$$
$$
E_0(1,2):= (\sgn_2)^{\ot |d|}\ot \id_1[d-1]=\mbox{span}\left\langle
\Ba{c}\resizebox{6mm}{!}{  \xy
(0,-7)*{_{1\hspace{5mm} 2}},
(0,5)*{}="U1";
(0,0)*{\bu}="o";
(-3,-5)*{}="D1";
(3,-5)*{}="D2";
\ar @{->} "D1";"o" <0pt>
\ar @{->} "D2";"o" <0pt>
\ar @{<-} "U1";"o" <0pt>
 \endxy}
 \Ea
=(-1)^{d}
\Ba{c}\resizebox{6mm}{!}{  \xy
(0,-7)*{_{2\hspace{5mm} 1}},
(0,5)*{}="U1";
(0,0)*{\bu}="o";
(-3,-5)*{}="D1";
(3,-5)*{}="D2";
\ar @{->} "D1";"o" <0pt>
\ar @{->} "D2";"o" <0pt>
\ar @{<-} "U1";"o" <0pt>
 \endxy}
 \Ea
\right\rangle
$$
by the ideal generated by the following relations
\Beq\label{3: R for LieB}
\cR:\left\{
\Ba{c}
\displaystyle
 \Ba{c}\resizebox{8.4mm}{!}{
\begin{xy}
 <0mm,0mm>*{\bu};<0mm,0mm>*{}**@{},
 <0mm,-0.49mm>*{};<0mm,-3.0mm>*{}**@{-},
 <0.49mm,0.49mm>*{};<1.9mm,1.9mm>*{}**@{-},
 <-0.5mm,0.5mm>*{};<-1.9mm,1.9mm>*{}**@{-},
 <-2.3mm,2.3mm>*{\bu};<-2.3mm,2.3mm>*{}**@{},
 <-1.8mm,2.8mm>*{};<0mm,4.9mm>*{}**@{-},
 <-2.8mm,2.9mm>*{};<-4.6mm,4.9mm>*{}**@{-},
   <0.49mm,0.49mm>*{};<2.7mm,2.3mm>*{^3}**@{},
   <-1.8mm,2.8mm>*{};<0.4mm,5.3mm>*{^2}**@{},
   <-2.8mm,2.9mm>*{};<-5.1mm,5.3mm>*{^1}**@{},
 \end{xy}}\Ea
 +\hspace{-2mm}
  \Ba{c}\resizebox{8.4mm}{!}{
\begin{xy}
 <0mm,0mm>*{\bu};<0mm,0mm>*{}**@{},
 <0mm,-0.49mm>*{};<0mm,-3.0mm>*{}**@{-},
 <0.49mm,0.49mm>*{};<1.9mm,1.9mm>*{}**@{-},
 <-0.5mm,0.5mm>*{};<-1.9mm,1.9mm>*{}**@{-},
 <-2.3mm,2.3mm>*{\bu};<-2.3mm,2.3mm>*{}**@{},
 <-1.8mm,2.8mm>*{};<0mm,4.9mm>*{}**@{-},
 <-2.8mm,2.9mm>*{};<-4.6mm,4.9mm>*{}**@{-},
   <0.49mm,0.49mm>*{};<2.7mm,2.3mm>*{^2}**@{},
   <-1.8mm,2.8mm>*{};<0.4mm,5.3mm>*{^1}**@{},
   <-2.8mm,2.9mm>*{};<-5.1mm,5.3mm>*{^3}**@{},
 \end{xy}}\Ea
 +\hspace{-2mm}
  \Ba{c}\resizebox{8.4mm}{!}{
\begin{xy}
 <0mm,0mm>*{\bu};<0mm,0mm>*{}**@{},
 <0mm,-0.49mm>*{};<0mm,-3.0mm>*{}**@{-},
 <0.49mm,0.49mm>*{};<1.9mm,1.9mm>*{}**@{-},
 <-0.5mm,0.5mm>*{};<-1.9mm,1.9mm>*{}**@{-},
 <-2.3mm,2.3mm>*{\bu};<-2.3mm,2.3mm>*{}**@{},
 <-1.8mm,2.8mm>*{};<0mm,4.9mm>*{}**@{-},
 <-2.8mm,2.9mm>*{};<-4.6mm,4.9mm>*{}**@{-},
   <0.49mm,0.49mm>*{};<2.7mm,2.3mm>*{^1}**@{},
   <-1.8mm,2.8mm>*{};<0.4mm,5.3mm>*{^3}**@{},
   <-2.8mm,2.9mm>*{};<-5.1mm,5.3mm>*{^2}**@{},
 \end{xy}}\Ea
 =0
 \ , \ \ \ \ \
\Ba{c}\resizebox{8.4mm}{!}{ \begin{xy}
 <0mm,0mm>*{\bu};<0mm,0mm>*{}**@{},
 <0mm,0.69mm>*{};<0mm,3.0mm>*{}**@{-},
 <0.39mm,-0.39mm>*{};<2.4mm,-2.4mm>*{}**@{-},
 <-0.35mm,-0.35mm>*{};<-1.9mm,-1.9mm>*{}**@{-},
 <-2.4mm,-2.4mm>*{\bu};<-2.4mm,-2.4mm>*{}**@{},
 <-2.0mm,-2.8mm>*{};<0mm,-4.9mm>*{}**@{-},
 <-2.8mm,-2.9mm>*{};<-4.7mm,-4.9mm>*{}**@{-},
    <0.39mm,-0.39mm>*{};<3.3mm,-4.0mm>*{^3}**@{},
    <-2.0mm,-2.8mm>*{};<0.5mm,-6.7mm>*{^2}**@{},
    <-2.8mm,-2.9mm>*{};<-5.2mm,-6.7mm>*{^1}**@{},
 \end{xy}}\Ea
 +\hspace{-2mm}
 \Ba{c}\resizebox{8.4mm}{!}{ \begin{xy}
 <0mm,0mm>*{\bu};<0mm,0mm>*{}**@{},
 <0mm,0.69mm>*{};<0mm,3.0mm>*{}**@{-},
 <0.39mm,-0.39mm>*{};<2.4mm,-2.4mm>*{}**@{-},
 <-0.35mm,-0.35mm>*{};<-1.9mm,-1.9mm>*{}**@{-},
 <-2.4mm,-2.4mm>*{\bu};<-2.4mm,-2.4mm>*{}**@{},
 <-2.0mm,-2.8mm>*{};<0mm,-4.9mm>*{}**@{-},
 <-2.8mm,-2.9mm>*{};<-4.7mm,-4.9mm>*{}**@{-},
    <0.39mm,-0.39mm>*{};<3.3mm,-4.0mm>*{^2}**@{},
    <-2.0mm,-2.8mm>*{};<0.5mm,-6.7mm>*{^1}**@{},
    <-2.8mm,-2.9mm>*{};<-5.2mm,-6.7mm>*{^3}**@{},
 \end{xy}}\Ea
 +\hspace{-2mm}
 \Ba{c}\resizebox{8.4mm}{!}{ \begin{xy}
 <0mm,0mm>*{\bu};<0mm,0mm>*{}**@{},
 <0mm,0.69mm>*{};<0mm,3.0mm>*{}**@{-},
 <0.39mm,-0.39mm>*{};<2.4mm,-2.4mm>*{}**@{-},
 <-0.35mm,-0.35mm>*{};<-1.9mm,-1.9mm>*{}**@{-},
 <-2.4mm,-2.4mm>*{\bu};<-2.4mm,-2.4mm>*{}**@{},
 <-2.0mm,-2.8mm>*{};<0mm,-4.9mm>*{}**@{-},
 <-2.8mm,-2.9mm>*{};<-4.7mm,-4.9mm>*{}**@{-},
    <0.39mm,-0.39mm>*{};<3.3mm,-4.0mm>*{^1}**@{},
    <-2.0mm,-2.8mm>*{};<0.5mm,-6.7mm>*{^3}**@{},
    <-2.8mm,-2.9mm>*{};<-5.2mm,-6.7mm>*{^2}**@{},
 \end{xy}}\Ea =0
\vspace{2mm} \\
 \Ba{c}\resizebox{6mm}{!}{\begin{xy}
 <0mm,2.47mm>*{};<0mm,0.12mm>*{}**@{-},
 <0.5mm,3.5mm>*{};<2.2mm,5.2mm>*{}**@{-},
 <-0.48mm,3.48mm>*{};<-2.2mm,5.2mm>*{}**@{-},
 <0mm,3mm>*{\bu};<0mm,3mm>*{}**@{},
  <0mm,-0.8mm>*{\bu};<0mm,-0.8mm>*{}**@{},
<-0.39mm,-1.2mm>*{};<-2.2mm,-3.5mm>*{}**@{-},
 <0.39mm,-1.2mm>*{};<2.2mm,-3.5mm>*{}**@{-},
     <0.5mm,3.5mm>*{};<2.8mm,5.7mm>*{^2}**@{},
     <-0.48mm,3.48mm>*{};<-2.8mm,5.7mm>*{^1}**@{},
   <0mm,-0.8mm>*{};<-2.7mm,-5.2mm>*{^1}**@{},
   <0mm,-0.8mm>*{};<2.7mm,-5.2mm>*{^2}**@{},
\end{xy}}\Ea
+(-1)^{d}
\Ba{c}\resizebox{7mm}{!}{\begin{xy}
 <0mm,-1.3mm>*{};<0mm,-3.5mm>*{}**@{-},
 <0.38mm,-0.2mm>*{};<2.0mm,2.0mm>*{}**@{-},
 <-0.38mm,-0.2mm>*{};<-2.2mm,2.2mm>*{}**@{-},
<0mm,-0.8mm>*{\bu};<0mm,0.8mm>*{}**@{},
 <2.4mm,2.4mm>*{\bu};<2.4mm,2.4mm>*{}**@{},
 <2.77mm,2.0mm>*{};<4.4mm,-0.8mm>*{}**@{-},
 <2.4mm,3mm>*{};<2.4mm,5.2mm>*{}**@{-},
     <0mm,-1.3mm>*{};<0mm,-5.3mm>*{^1}**@{},
     <2.5mm,2.3mm>*{};<5.1mm,-2.6mm>*{^2}**@{},
    <2.4mm,2.5mm>*{};<2.4mm,5.7mm>*{^2}**@{},
    <-0.38mm,-0.2mm>*{};<-2.8mm,2.5mm>*{^1}**@{},
    \end{xy}}\Ea
  +
\Ba{c}\resizebox{7mm}{!}{\begin{xy}
 <0mm,-1.3mm>*{};<0mm,-3.5mm>*{}**@{-},
 <0.38mm,-0.2mm>*{};<2.0mm,2.0mm>*{}**@{-},
 <-0.38mm,-0.2mm>*{};<-2.2mm,2.2mm>*{}**@{-},
<0mm,-0.8mm>*{\bu};<0mm,0.8mm>*{}**@{},
 <2.4mm,2.4mm>*{\bu};<2.4mm,2.4mm>*{}**@{},
 <2.77mm,2.0mm>*{};<4.4mm,-0.8mm>*{}**@{-},
 <2.4mm,3mm>*{};<2.4mm,5.2mm>*{}**@{-},
     <0mm,-1.3mm>*{};<0mm,-5.3mm>*{^2}**@{},
     <2.5mm,2.3mm>*{};<5.1mm,-2.6mm>*{^1}**@{},
    <2.4mm,2.5mm>*{};<2.4mm,5.7mm>*{^2}**@{},
    <-0.38mm,-0.2mm>*{};<-2.8mm,2.5mm>*{^1}**@{},
    \end{xy}}\Ea
  + (-1)^{d}
\Ba{c}\resizebox{7mm}{!}{\begin{xy}
 <0mm,-1.3mm>*{};<0mm,-3.5mm>*{}**@{-},
 <0.38mm,-0.2mm>*{};<2.0mm,2.0mm>*{}**@{-},
 <-0.38mm,-0.2mm>*{};<-2.2mm,2.2mm>*{}**@{-},
<0mm,-0.8mm>*{\bu};<0mm,0.8mm>*{}**@{},
 <2.4mm,2.4mm>*{\bu};<2.4mm,2.4mm>*{}**@{},
 <2.77mm,2.0mm>*{};<4.4mm,-0.8mm>*{}**@{-},
 <2.4mm,3mm>*{};<2.4mm,5.2mm>*{}**@{-},
     <0mm,-1.3mm>*{};<0mm,-5.3mm>*{^2}**@{},
     <2.5mm,2.3mm>*{};<5.1mm,-2.6mm>*{^1}**@{},
    <2.4mm,2.5mm>*{};<2.4mm,5.7mm>*{^1}**@{},
    <-0.38mm,-0.2mm>*{};<-2.8mm,2.5mm>*{^2}**@{},
    \end{xy}}\Ea
 +
\Ba{c}\resizebox{7mm}{!}{\begin{xy}
 <0mm,-1.3mm>*{};<0mm,-3.5mm>*{}**@{-},
 <0.38mm,-0.2mm>*{};<2.0mm,2.0mm>*{}**@{-},
 <-0.38mm,-0.2mm>*{};<-2.2mm,2.2mm>*{}**@{-},
<0mm,-0.8mm>*{\bu};<0mm,0.8mm>*{}**@{},
 <2.4mm,2.4mm>*{\bu};<2.4mm,2.4mm>*{}**@{},
 <2.77mm,2.0mm>*{};<4.4mm,-0.8mm>*{}**@{-},
 <2.4mm,3mm>*{};<2.4mm,5.2mm>*{}**@{-},
     <0mm,-1.3mm>*{};<0mm,-5.3mm>*{^1}**@{},
     <2.5mm,2.3mm>*{};<5.1mm,-2.6mm>*{^2}**@{},
    <2.4mm,2.5mm>*{};<2.4mm,5.7mm>*{^1}**@{},
    <-0.38mm,-0.2mm>*{};<-2.8mm,2.5mm>*{^2}**@{},
    \end{xy}}\Ea=0.
    \Ea
\right.
\Eeq
where the vertices are assumed to be ordered in such a way that the ones on the bottom come first, and edges are implicitly directed from bottom to the top.
Its arbitrary representation,
$$
\rho: \LBd \lon \cE nd_V,
$$
in a dg vector space $V$ is uniquely determined
by the values of $\rho$ on the generators,
$$
\rho \left(
 \begin{xy}
 <0mm,-0.55mm>*{};<0mm,-2.5mm>*{}**@{-},
 <0.5mm,0.5mm>*{};<2.2mm,2.2mm>*{}**@{-},
 <-0.48mm,0.48mm>*{};<-2.2mm,2.2mm>*{}**@{-},
 <0mm,0mm>*{\bu};<0mm,0mm>*{}**@{},
 \end{xy}\right)=: \Delta, \  \ \ \
\rho\left(
 \begin{xy}
 <0mm,0.66mm>*{};<0mm,3mm>*{}**@{-},
 <0.39mm,-0.39mm>*{};<2.2mm,-2.2mm>*{}**@{-},
 <-0.35mm,-0.35mm>*{};<-2.2mm,-2.2mm>*{}**@{-},
 <0mm,0mm>*{\bu};<0mm,0mm>*{}**@{},
 \end{xy}
 \right)=: [\ ,\ ],
$$
which make  $V$ into a degree $d$ Lie bialgebra in the above sense. The properad $\LB_1$ corresponds to the classical properad of Lie bialgebras which is often denoted simply by $\LB$.  The homotopy theory of (possibly, infinite-dimensional) degree $d$ Lie bialgebras is highly non-trivial --- it is controlled by M.\ Kontsevich's graph complex $\GC_{2d}^{\geq 2}$ via its oriented incarnation $\OGC_{2d+1}$
(see \cite{MW1} for the proof); in particular, in the most important case $d=1$  the famous and mysterious Grothendieck-Teichm\"uller group $GRT$ \cite{Dr2} acts on the genus completed version of the properad $\LB$ as homotopy non-trivial automorphisms. We show in this paper that this action commutes with a natural reversing direction involution of that properad.

\sip

The minimal resolution, $\HoLBd$, of the properad $\LBd$ is a dg free properad,
$$
\HoLBd:=\cF ree \left\langle E\right\rangle,
$$
 generated by an $\bS$-bimodule
 $
 E=\{E(m,n)\}_{m,n\geq 1, m+n\geq 3}
$
  given by
 $$
{E}(m,n):=\sgn_m^{\ot |d|}\ot \sgn_n^{|d|}[d(m+n-2)-1]\equiv\text{span}\left\langle\hspace{-1mm}
\Ba{c}\resizebox{17mm}{!}  {\xy
(0,9)*{^{\sigma(1)\ \ \ \sigma(2)\ \hspace{5mm}\ \ \sigma(m)}},
(0,-9)*{_{\tau(1)\ \ \ \tau(2)\ \hspace{5mm}\ \ \tau(n)}},
(0,4.9)*+{...},
(0,-4.9)*+{...},
(0,0)*{\bu}="o",
(-7,7)*{}="1",
(-3,7.0)*{}="2",
(3,7)*{}="3",
(7,7.0)*{}="4",
(-3,-7)*{}="5",
(3,-7)*{}="6",
(7,-7)*{}="7",
(-7,-7)*{}="8",
\ar @{->} "o";"1" <0pt>
\ar @{->} "o";"2" <0pt>
\ar @{->} "o";"3" <0pt>
\ar @{->} "o";"4" <0pt>
\ar @{<-} "o";"5" <0pt>
\ar @{<-} "o";"6" <0pt>
\ar @{<-} "o";"7" <0pt>
\ar @{<-} "o";"8" <0pt>
\endxy}\Ea
=(-1)^{d(|\sigma|+|\tau|)}
\Ba{c}\resizebox{12mm}{!}  {\xy
(0,9)*{^{1\ \ \ 2\ \hspace{7mm}\ \ m}},
(0,-9)*{_{1\ \ \ 2\ \hspace{7mm}\ \ n}},
(0,4.9)*+{...},
(0,-4.9)*+{...},
(0,0)*{\bu}="o",
(-7,7)*{}="1",
(-3,7.0)*{}="2",
(3,7)*{}="3",
(7,7.0)*{}="4",
(-3,-7)*{}="5",
(3,-7)*{}="6",
(7,-7)*{}="7",
(-7,-7)*{}="8",
\ar @{->} "o";"1" <0pt>
\ar @{->} "o";"2" <0pt>
\ar @{->} "o";"3" <0pt>
\ar @{->} "o";"4" <0pt>
\ar @{<-} "o";"5" <0pt>
\ar @{<-} "o";"6" <0pt>
\ar @{<-} "o";"7" <0pt>
\ar @{<-} "o";"8" <0pt>
\endxy}\Ea
 \right\rangle_{ \forall \sigma\in \bS_m \atop \forall\tau\in \bS_n}.
$$
The differential in $\HoLBd$ is given on the generators by
 \Beq\label{3: d in HoLBd_infty}
\delta
\Ba{c}\resizebox{11mm}{!}  {\xy
(0,9)*{^{1\ \ \ 2\ \hspace{7mm}\ \ m}},
(0,-9)*{_{1\ \ \ 2\ \hspace{7mm}\ \ n}},
(0,4.9)*+{...},
(0,-4.9)*+{...},
(0,0)*{\bu}="o",
(-7,7)*{}="1",
(-3,7.0)*{}="2",
(3,7)*{}="3",
(7,7.0)*{}="4",
(-3,-7)*{}="5",
(3,-7)*{}="6",
(7,-7)*{}="7",
(-7,-7)*{}="8",
\ar @{->} "o";"1" <0pt>
\ar @{->} "o";"2" <0pt>
\ar @{->} "o";"3" <0pt>
\ar @{->} "o";"4" <0pt>
\ar @{<-} "o";"5" <0pt>
\ar @{<-} "o";"6" <0pt>
\ar @{<-} "o";"7" <0pt>
\ar @{<-} "o";"8" <0pt>
\endxy}\Ea
\ \ = \ \
 \sum_{[m]=I_1\sqcup I_2\atop
 {|I_1|\geq 0, |I_2|\geq 1}}
 \sum_{[n]=J_1\sqcup J_2\atop
 {|J_1|\geq 1, |J_2|\geq 0}
}\hspace{0mm}
(-1)^{d(\# J_1 + \# I_1\# J_2 + \sigma(I_1,I_2) + \sigma(J_1,J_2))}
\Ba{c}\resizebox{18mm}{!}  {\xy
(-2,9.5)*{\overbrace{\hspace{12mm}}^{I_1}},
(0,-10)*{\underbrace{\hspace{16mm}}_{J_1}},
(10.4,19.0)*{\overbrace{\hspace{12mm}}^{I_2}},
(11,-1)*{\underbrace{\hspace{12mm}}_{J_2}},
(0,4.9)*+{...},
(0,-4.9)*+{...},
(12,14)*+{...},
(12.6,3)*+{...},
(0,0)*{\bu}="o",
(-7,7)*{}="1",
(-3,7.0)*{}="2",
(3,7)*{}="3",
(11,9.0)*{\bu}="O",
(-3,-7)*{}="5",
(3,-7)*{}="6",
(7,-7)*{}="7",
(-7,-7)*{}="8",
(5,16)*{}="U1",
(9,16)*{}="U2",
(16,16)*{}="U3",
(6,2)*{}="D1",
(10,2)*{}="D2",
(16,2)*{}="D3",
\ar @{->} "o";"1" <0pt>
\ar @{->} "o";"2" <0pt>
\ar @{->} "o";"3" <0pt>
\ar @{->} "o";"O" <0pt>
\ar @{<-} "o";"5" <0pt>
\ar @{<-} "o";"6" <0pt>
\ar @{<-} "o";"7" <0pt>
\ar @{<-} "o";"8" <0pt>
\ar @{->} "O";"U1" <0pt>
\ar @{->} "O";"U2" <0pt>
\ar @{->} "O";"U3" <0pt>
\ar @{<-} "O";"D1" <0pt>
\ar @{<-} "O";"D2" <0pt>
\ar @{<-} "O";"D3" <0pt>
\endxy}\Ea
\Eeq
where the vertices on the r.h.s.\ are ordered in such a way that the lowest one (with respect to the given flow on edges) comes first.
In the most important special case $d=1$ the properad $\HoLB_1$ is abbreviated with $\HoLB$.

\subsection{Orientation of graphs from $\HoLBd$}
A generic element $\Ga$ of $\HoLBd$ (in particular, of $\LBd$) can be constructed iteratively from generators by gluing in-legs of one generator to some (or all) out-legs of other generator(s). To define such a composed graph uniquely, it must be equipped with one of the two possible {\em orientations}, $or$ or $or^{op}$, which identifies $(\Ga, or)$ with $-(\Ga,or^{op})$.

If $d$ is even, every vertex of $\Ga$ is odd so that an orientation of $\Ga$
is defined as one of  two possible unit vectors of the  1-dimensional real space
$\det \left(\R V(\Ga) \right)$ where $V(\Ga)$ stands for the set of vertices of $\Ga$ and $\R  V(\Ga)$ its linear span over $\R$. Note that $or$ does not depend on the particular choice of directions on edges.

If $d$ is odd, every $(m,n)$-generator $\Ba{c}\resizebox{7mm}{!}  {\xy
(0,9)*{^{1\ \ \ 2\ \hspace{7mm}\ \ m}},
(0,-9)*{_{1\ \ \ 2\ \hspace{7mm}\ \ n}},
(0,4.9)*+{...},
(0,-4.9)*+{...},
(0,0)*{\bu}="o",
(-7,7)*{}="1",
(-3,7.0)*{}="2",
(3,7)*{}="3",
(7,7.0)*{}="4",
(-3,-7)*{}="5",
(3,-7)*{}="6",
(7,-7)*{}="7",
(-7,-7)*{}="8",
\ar @{->} "o";"1" <0pt>
\ar @{->} "o";"2" <0pt>
\ar @{->} "o";"3" <0pt>
\ar @{->} "o";"4" <0pt>
\ar @{<-} "o";"5" <0pt>
\ar @{<-} "o";"6" <0pt>
\ar @{<-} "o";"7" <0pt>
\ar @{<-} "o";"8" <0pt>
\endxy}\Ea
$ can be understood as an odd vertex $v$ (of cohomological degree $1+2d$) with $m+n$ odd {\it half-edges}\, attached (each of degree $-d$), $h_1,\ldots, h_n$ incoming half-edges and $h^1,\ldots, h^m$ outgoing ones (in particular, each of the two degree zero generators of $\LB$ is best understood --- in the context of graph orientations --- as a degree $3$ vertex with three degree $-1$ half-edges attached). To equip such an $(m,n)$-corolla with a well-defined sign as an element of $\HoLBd$, we have to choose an ordering of half-edges and the vertex, say, as follows,
$$
h_1\wedge\ldots \wedge h_n\wedge v \wedge h^1,\ldots, h^n.
$$
The differential $\delta$ acts on this orientation by the formula
$$
\delta(h_1\wedge\ldots \wedge h_n\wedge v \wedge h^1\wedge \ldots\wedge h^n)=
(-1)^n h_1\wedge\ldots \wedge h_n\wedge (\delta v) \wedge h^1,\ldots, h^n=(-1)^n h_1\wedge\ldots \wedge h_n\wedge v'\wedge h^0 \wedge h_0\wedge v'' \wedge h^1,\ldots, h^n
$$
which explains the signs in (\ref{3: d in HoLBd_infty}).
 A generic element $\Ga\in \HoLBd$, $d\in 2\Z +1$, is built from such corolla by gluing outgoing half-edges of one corolla to the ingoing half-edges of another corolla so that each internal edge $e$ of $\Ga$ has an associated set $H(e)$ of two-half edges; the remaining half-edges constitute the set $L(\Ga)=L_{in}(\Ga)\sqcup L_{out}$ of legs of $\Ga$ which are labelled. The orientation of $\Ga$ is then uniquely determined by a unit vector of the  1-dimensional Euclidean real space
$$
\det\left(\R  H(\Ga) \oplus \R  V(\Ga) \right)\simeq \det\left(\R  H(\Ga) \right)\ot \det \left(\R  V(\Ga) \right)
$$
where $H(\Ga)$ is the set of all half-edges of $\Ga$, and $V(\Ga)$ the set of vertices. Grouping the set $H(\Ga)$ into internal edges (whose set is denoted by $E(\Ga)$) and legs, we obtain a canonical isomorphism
$$
\det\left(\R H(\Ga) \right)=\ot_{e\in E(\Ga)} \det \R H(e)\ot \det \left(\R  L_{in}(\Ga) \ot \right)\ot \det \left(\R  L_{in}(\Ga)\right)
$$
As $\Ga$ is directed and every leg is labelled, the latter Euclidean space comes equipped canonically with a distinguished unit vector so that an orientation $or$ of a generic element $\Ga\in \HoLBd$ for $d$ odd is identified again with a unit vector in $\det \left(\R V(\Ga)\right)$ (as in the case $d$ is even).

\subsection{Direction involution of $\HoLBd$}
For a graph
$\Ga \in \HoLB_{d}$ with orientation $or\in \det \left(\R V(\Ga) \right)$  we denote by  $\st{\Ga}$ an element $\HoLB_d$ obtained from $\Ga$ by reversing  simultaneously the direction of each internal edge and leg of $\Ga$ while {\it keeping its orientation $or$ unchanged}, i.e. the vertices of
$\st{\Ga}$ are ordered in exactly the same way as the vertices of $\Ga$.
The map $\Ga\rar \st{\Ga}$  squares to the identity map, but it does not commute with the differential $\delta$. Indeed, reversing arrows of both sides in (\ref{3: d in HoLBd_infty})
while keeping the ordering of vertices on the r.h.s.\ unchanged gives a graph with the orientation opposite to the one which is tacitly assumed in (\ref{3: d in HoLBd_infty}), i.e.\ we get eventually the wrong sign in the r.h.s.\ after the transformation. Hence
we have to consider a slightly different map given by
\Beq\label{2: involution on HoLBd}
\Ba{rccc}
\imath: & \HoLBd & \lon & \HoLBd\\
        & \Ga & \lonm & (-1)^{(\#V(\Ga)+1)}\st{\Ga}
   %
\Ea
\Eeq
which does commute with the differential and hence provides us with a genuine {\em direction involution}\,   of the complex $(\HoLBd, \delta)$. For example,
$$
\imath\left(\Ba{c}\resizebox{6mm}{!}{  \xy
(0,-7)*{_{1\hspace{5mm} 2}},
(0,5)*{}="U1";
(0,0)*{\bu}="o";
(-3,-5)*{}="D1";
(3,-5)*{}="D2";
\ar @{->} "D1";"o" <0pt>
\ar @{->} "D2";"o" <0pt>
\ar @{<-} "U1";"o" <0pt>
 \endxy}
 \Ea
 \right)=\Ba{c}\resizebox{6mm}{!}{  \xy
(0,7)*{_{1\hspace{5mm} 2}},
(0,-5)*{}="U1";
(0,0)*{\bu}="o";
(-3,5)*{}="D1";
(3,5)*{}="D2";
\ar @{<-} "D1";"o" <0pt>
\ar @{<-} "D2";"o" <0pt>
\ar @{->} "U1";"o" <0pt>
 \endxy}
 \Ea, \ \ \ \
 \imath\left(
 \Ba{c}\resizebox{6mm}{!}{  \xy
(0,7)*{_{1\hspace{5mm} 2}},
(0,-5)*{}="U1";
(0,0)*{\bu}="o";
(-3,5)*{}="D1";
(3,5)*{}="D2";
\ar @{<-} "D1";"o" <0pt>
\ar @{<-} "D2";"o" <0pt>
\ar @{->} "U1";"o" <0pt>
 \endxy}
 \Ea\right)=
 \Ba{c}\resizebox{6mm}{!}{  \xy
(0,-7)*{_{1\hspace{5mm} 2}},
(0,5)*{}="U1";
(0,0)*{\bu}="o";
(-3,-5)*{}="D1";
(3,-5)*{}="D2";
\ar @{->} "D1";"o" <0pt>
\ar @{->} "D2";"o" <0pt>
\ar @{<-} "U1";"o" <0pt>
 \endxy}
 \Ea
$$
which corresponds to the duality transformation, $V\rar V^*$, of a Lie bialgebra $V$.

{\Large
\section{\bf Direction involution of oriented graph complexes}
}

\subsection{Reminder on oriented graph complexes (cf. \cite{Wi2,MultiOriented})}
\label{s:OGC}
Let $\bar\mV_v\bar\mE_e\mO\grac^{2}$ be the set of directed connected graphs $\Ga$
with  $v>0$ (labelled by $1,\dots,v$) at least 2-valent vertices, with  $e\geq 0$ (labelled by $1,\dots,v$) directed edges, and with no  passing vertices (i.e., 2-valent vertices with one incoming and one outgoing edge $\EEpassing$), and with no closed directed paths of directed edges.
For $d\in\Z$, let
\begin{equation}
\bar\mV_v\bar\mE_e\mO\G_d:=\left\langle\bar\mV_v\bar\mE_e\mO\grac^{2}\right\rangle[(d-1)e-dv+d]
\end{equation}
be the vector space of degree shifted formal linear combinations.

There is a natural right action of the group $\sym_v\times \sym_e$ on $\bar\mV_v\bar\mE_e\mO\G_d$, where $\sym_v$ permutes vertices and $\sym_e$ permutes edges.
Let $\sgn_v$ and $\sgn_e$ be one-dimensional representations of $\sym_v$, respectively $\sym_e$, where the odd permutation reverses the sign.
Let us consider the space of coinvariants:
\begin{equation}
\label{eq:VED}
\mV_v\mE_e\mO\G_{d}:=\left\{
\begin{array}{ll}
\left(\bar\mV_v\bar\mE_e\mO\G_d\otimes\sgn_e\right)_{\sym_v\times\sym_e}
\qquad&\text{for $d$ even,}\\
\left(\bar\mV_v\bar\mE_e\mO\G_d\otimes\sgn_v\right)_{\sym_v\times\sym_e}
\qquad&\text{for $d$ odd,}
\end{array}
\right.
\end{equation}
Every graph $\Ga\in \mV_v\mE_e\OG_d$ comes equipped with an orientation, $or$, which is given by an ordering of edges (up to an even permutation) in the case $d\in 2\Z$, or, respectively, by an ordering of vertices (up to an even permutation) in the case $d\in 2\Z+1$; equivalently, $or$ can be identified with a unit vector of the 1-dimensional Euclidean vector space $\det(\R E(\Ga))$, respectively $\det(\R V(\Ga))$.

The underlying vector space of
the \emph{oriented graph complex} is given by
\begin{equation}
\OG_d:=\bigoplus_{v\geq 1,e\geq 0}\mV_v\mE_e\mO\G_d.
\end{equation}
The differential $\p$ acts by edge contraction:
\begin{equation}
\p(\Gamma)=\sum_{a\in E(\Gamma)} \pm\Gamma/a
\end{equation}
where $E(\Gamma)$ is the set of edges of $\Gamma$ and $\Gamma/a$ is the graph produced from $\Gamma$ by contracting edge $a$ and merging its end vertices.
If a directed cycle or a passing vertex is produced, we consider the result to be zero.
To precisely define the sign, we agree on the following convention: as a representative of a coinvariant $\Gamma\in\mV_v\mE_e\mO\G_d$ we take $\bar\Gamma\in\bar\mV_v\bar\mE_e\mO\G_d$ where the edge being deleted is the last one (labelled by $e$), heading towards the last vertex (labelled by $v$). Then contracting the edge deletes that edge and the vertex it heads to, and leaves all the other labels unchanged, with the sign +.


\subsection{Skeleton version}

In \cite[Subsection 3.4]{MultiSourced} the isomorphic skeleton version of oriented graph complex $\mO^{sk}\G_d$ is introduced.
It is the complex of essentially the same graphs where vertices that are at least 3-valent are seen as \emph{skeleton vertices}, and the structure of edges and 2-valent vertices between them is seen as a \emph{skeleton edge}.
The set of such skeleton edges is as following
\begin{equation}
\label{eq:SkE}
\sigma_d^\infty=\{
\Ed,\;
\dE,\;
\EdE,\;
\dEd,\;
\begin{tikzpicture}[baseline=-.65ex,scale=.5]
 \node[nil] (a) at (0,0) {};
 \node[int] (b) at (1,0) {};
 \node[int] (c) at (2,0) {};
 \node[nil] (d) at (3,0) {};
 \draw (a) edge[-latex] (b);
 \draw (b) edge[latex-] (c);
 \draw (c) edge[-latex] (d);
\end{tikzpicture},\;
\begin{tikzpicture}[baseline=-.65ex,scale=.5]
 \node[nil] (a) at (0,0) {};
 \node[int] (b) at (1,0) {};
 \node[int] (c) at (2,0) {};
 \node[nil] (d) at (3,0) {};
 \draw (a) edge[latex-] (b);
 \draw (b) edge[-latex] (c);
 \draw (c) edge[latex-] (d);
\end{tikzpicture},\;
\begin{tikzpicture}[baseline=-.65ex,scale=.5]
 \node[nil] (a) at (0,0) {};
 \node[int] (b) at (1,0) {};
 \node[int] (c) at (2,0) {};
 \node[int] (d) at (3,0) {};
 \node[nil] (e) at (4,0) {};
 \draw (a) edge[-latex] (b);
 \draw (b) edge[latex-] (c);
 \draw (c) edge[-latex] (d);
 \draw (d) edge[latex-] (e);
\end{tikzpicture},\;
\begin{tikzpicture}[baseline=-.65ex,scale=.5]
 \node[nil] (a) at (0,0) {};
 \node[int] (b) at (1,0) {};
 \node[int] (c) at (2,0) {};
 \node[int] (d) at (3,0) {};
 \node[nil] (e) at (4,0) {};
 \draw (a) edge[latex-] (b);
 \draw (b) edge[-latex] (c);
 \draw (c) edge[latex-] (d);
 \draw (d) edge[-latex] (e);
\end{tikzpicture},\;
\dots
\}.
\end{equation}
Note that the differential on $\left(\mO^{sk}\G_d,\p\right)$ has two parts:
\begin{equation}
\p=\p_C\pm \p_E
\end{equation}
where the \emph{core differential} $\p_C$ contracts skeleton edges $\Ed$ and $\dE$, and the \emph{edge differential} $\p_E$ an original edge within a longer skeleton edge, thus changing the type of the skeleton edge.
We have the following result.

\begin{prop}[{\cite[Proposition 3.9, Equation (76)]{MultiSourced}}]
\label{prop:sk}
There is an isomorphism of complexes
$$
\left(\mO^{sk}\G_d,\p\right)\rightarrow
\left(\mO\G_d,\p\right).
$$
\end{prop}

The set of skeleton edges \eqref{eq:SkE} can be reduced without changing the homology. Let
\begin{equation}
\label{eq:rSkE}
\sigma_d^1=\{\Ed,\;\dE,\;\Ess\},
\end{equation}
where
\begin{equation}
\Ess:=\frac{1}{2}\left(\EdE-\dEd\right).
\end{equation}
Note also that
\begin{equation}
\begin{tikzpicture}[baseline=-.65ex,scale=.5]
 \node[nil] (a) at (0,0) {};
 \node[nil] (c) at (1.4,0) {};
 \draw (a) edge[<-] (c);
 \draw (.7,.15) edge (.7,-.15);
\end{tikzpicture}
=-(-1)^d\Ess.
\end{equation}
Graph complex $\mO^1\G_d$ is the subcomplex of $\mO^{sk}\G_d$ spanned by graphs whose skeleton edges belong to $\sigma_d^1$.
Here, the core differential $\p_C$ contracts $\Ed$ or $\dE$, and the edge differential $\p_E$ changes
\begin{equation}
\p_E(\Ess)=\Ed-(-1)^d\dE.
\end{equation}
We have the following result.

\begin{prop}[{\cite[Proposition 3.15 part 2.]{MultiSourced}}]
\label{prop:1}
The inclusion
$$
\left(\mO^1\G_d,\p\right)\hookrightarrow
\left(\mO^{sk}\G_d,\p\right)
$$
is a quasi-isomorphism.
\end{prop}

\subsection{Reversing directions of all edges} Let $\st{\Ga}$ be the graph obtained from a graph $\Ga\in \OG_d$ by reversing simultaneously the direction of every its edge while keeping its orientation $or$ unchanged (which makes sense as $V(\st{\Ga})=V(\Ga)$ and $E(\st{\Ga})=E(\Ga)$). The associated map
\Beq\label{3: involution on OG_d}
\Ba{rccc}
\imath: & \OG_d & \lon & \OG_d\\
        & \Ga & \lonm &
\left\{\Ba{cc} (-1)^{e+v+1}\st{\Ga}  & \text{for $d$ even} \\
(-1)^{v +1} \st{\Ga} & \text{for $d$ odd}
\Ea
\right.
\Ea
\Eeq
commutes with the differential and defines therefore an involution of $\OG_d$. The involution has two eigenvalues, $+1$ and $-1$, so that it decomposes the complex into a direct sum,
\begin{equation}
\OG_d= \OG_d^+ \oplus \OG_d^-,
\end{equation}
where, for example, the subcomplex $\OG_d^-$ is spanned by graphs $\Ga$ satisfying the condition
\Beq\label{3: relation on OG^-}
\st{\Ga}\;= \left\{\Ba{cc} (-1)^{e+v}\Ga  & \text{for $d$ even} \\
(-1)^{v} \Ga & \text{for $d$ odd}
\Ea
\right.
\Eeq

We shall show next that the complex  $\OG_d^-$ is acyclic.

\subsection{Skeleton version}

The involution $\iota$ can be defined on skeleton complex $\mO^1\G_d$.
First, let $\st{\Ga}$ be the graph obtained from a graph $\Ga\in \mO^1\G_d$ by reversing simultaneously the direction of every its edge of type $\Ed$ or $\dE$, while edges of type $\Ess$ are retained.
It is easy to check that $\st{\Ga}$ is mapped to $(-1)^{s}\st{\Ga}$ by the inclusion from Proposition \ref{prop:1}, where $s$ is the number of edges of type $\Ess$.
Therefore, if we define the involution $\iota$ on $\mO^1\G_d$ as
\Beq\label{3: involution on O1G_d}
\Ba{rccc}
\imath: & \mO^1\G_d & \lon & \mO^1\G_d\\
        & \Ga & \lonm &
\left\{\Ba{cc} (-1)^{e+v+s+1} \st{\Ga}  & \text{for $d$ even} \\
(-1)^{v +1} \st{\Ga} & \text{for $d$ odd,}
\Ea
\right.
\Ea
\Eeq
it will commute with the inclusion from Proposition \ref{prop:1}.

By analogy to the above, we decompose
$$
\mO^1\G_d= \mO^1\G_d^+ \oplus \mO^1\G_d^-,
$$
the skeleton complex into a direct sum of subcomplexes spanned by eigenvectors with eigenvalue $+1$ and, respectively $-1$.

\begin{prop}
\label{prop:skeleton}
The inclusions
$$
\left({\mO}^1\G_d^-,\p\right)\hookrightarrow
\left({\mO}\G_d^-,\p\right),\ \ \ \
\left({\mO}^1\G_d^+,\p\right)\hookrightarrow
\left({\mO}\G_d^+,\p\right)
$$
are quasi-isomorphisms.
\end{prop}
\begin{proof}
Straightforward from Propositions \ref{prop:sk} and \ref{prop:1} and the construction above.
\end{proof}

\subsection{Acyclicity} The strategy of our proof of the following statement resembles the one from \cite{MultiSourced}.

\begin{prop}
\label{prop:acyclic}
The complex  $\left({\mO}^1\G_d^-,\p\right)$  is acyclic.
\end{prop}
\begin{proof}
We set up the spectral sequence on the number of vertices, such that the first differential is the edge differential $\p_E$ that changes $\Ess\mapsto\Ed-(-1)^d\dE$. Standard splitting of complexes as the product of complexes with fixed loop number implies that the spectral sequence converges to the homology of the whole complex. So it is enough to prove that $\left({\mO}^{1}\G_d^-,\p_E\right)$ is acyclic.
Neither the differential $\p_E$ nor the involution $\iota$ change the total number of vertices $v$ and the total number $e$ of edges, so the latter complexes decomposes into a direct sum of complexes
$$
\left({\mO}^{1}\G_d^-,\p_E\right)=
\bigoplus_{v\geq 1,e\geq 0}\left( \mV_v\mE_e\mO^1\G_d^-, \p_E\right)
$$
spanned by graphs with {\em fixed}\,  numbers $e$ and $\nu$. By Maschke theorem, to prove
acyclicity of $\left( \mV_v\mE_e\mO^1\G_d^-, \p_E\right)$ it is enough to prove acyclicity of its version,
$$
\left({\mO}^1\bar\mV_v\bar\mE_e\G_d^-,\p_E\right),
$$
in which vertices and edges are distinguished, say labelled by integers from $[v]$ and $[e]$ respectively.

\sip

Let a \emph{core graph} of a graph $\Gamma\in {\mO}^1\bar\mV_v\bar\mE_e\G_d$ be the graph $\Phi$ gained from $\Gamma$ by forgetting the types of edges. We say that $\Gamma$ is of the \emph{shape} $\Phi$.

Neither differential $\p_E$ nor the involution $\iota$ changes the core graph, so both complexes ${\mO}^1\bar\mV_v\bar\mE_e\G_d$ and ${\mO}^1\bar\mV_v\bar\mE_e\G_d^-$ split as direct sums of complexes for the fixed core graph.
Let us fix a core graph $\Phi$ with $v$ labelled vertices and $e$ labelled edges. Let $\mO\Phi$ be the complex spanned by graphs of the form $\Phi$ whose edges are labelled by $\Ed$ in any direction, or $\Ess$ in predefined direction.
The involution $\iota$ acts as in \eqref{3: involution on O1G_d}. Let us recall it by an equivalent expression:
$$
\Ga  \lonm
-\left\{\Ba{cc} (-1)^{v+\text{number of edges of type}\Ed} \st{\Ga}  & \text{for $d$ even} \\
(-1)^{v} \st{\Ga} & \text{for $d$ odd,}
\Ea
\right.
$$
It leads to the decomposition to eigenspaces of eigenvalues $+1$ and $-1$:
 $$
 \mO\Phi =\mO\Phi^+ \oplus \mO\Phi^-.
 $$
The complex ${\mO}^1\bar\mV_v\bar\mE_e\G_d^-$ splits as a direct sum of $\mO\Phi^-$ for all possible core graphs $\Phi$, so it is enough to prove that $(\mO\Phi^-,\p)$ is acyclic.

In $\Phi$ we choose $v-1$ edges, say $a_1,\dots,a_{v-1}$, such that for every $i=1,\dots,v-1$ the edges $\{a_1,\dots, a_i\}$ form a sub-graph of $\Phi$ that is a tree. Clearly, $\{a_1,\dots, a_{v-1}\}$ forms a spanning tree.
For every $i=0,\dots,v-1$ we form graph complex ${\mO}\Phi_{i}$ as follows.

First we introduce a new type of edge, $\EE$. The differential $\p_E$ does not act on it.
$\mO\Phi_i$ is the complex spanned by graphs of the form $\Phi$ whose edges $a_1,\dots,a_i$ are labelled by $\EE$, and other edges are labelled by $\Ed$ in any direction, or $\Ess$ in predefined direction, such that there are no cycles.
Here, a \emph{cycle} is a sequence of vertices $x_0,x_1,\dots, x_p=x_0$ in $\Phi$ such that there is an edge from $x_j$ to $x_{j+1}$ labelled by $\EE$ or $\Ed$ in the direction from $x_j$ to $x_{j+1}$. Roughly speaking, edges $\EE$ are considered to go in both directions for the matter of cycles. An example of a graph is shown in Figure \ref{fig:complexSigma}.

\begin{figure}[H]
$$
\begin{tikzpicture}[baseline=1ex]
 \node[int] (a) at (-1,-.5) {};
 \node[int] (b) at (1,-.5) {};
 \node[int] (c) at (0,0) {};
 \node[int] (d) at (0,1.1) {};
 \draw (a) edge (b);
 \draw (a) edge (c);
 \draw (a) edge (d);
 \draw (b) edge (c);
 \draw (b) edge (d);
 \draw (c) edge[bend left=15] (d);
 \draw (c) edge[bend right=15] (d);
\end{tikzpicture}\quad\quad
\begin{tikzpicture}[baseline=1ex]
 \node[int] (a) at (-1,-.5) {};
 \node[int] (b) at (1,-.5) {};
 \node[int] (c) at (0,0) {};
 \node[int] (d) at (0,1.1) {};
 \draw (a) edge node[below] {$\scriptstyle a_3$} (c);
 \draw (b) edge node[below] {$\scriptstyle a_2$} (c);
 \draw (c) edge[bend left=15] node[right] {$\scriptstyle a_1$} (d);
\end{tikzpicture}\quad\quad
\begin{tikzpicture}[baseline=1ex]
 \node[int] (a) at (-1,-.5) {};
 \node[int] (b) at (1,-.5) {};
 \node[int] (c) at (0,0) {};
 \node[int] (d) at (0,1.1) {};
 \draw (a) edge[-latex] (b);
 \draw (a) edge[-latex] (c);
 \draw (a) edge[crossed,<-] (d);
 \draw (b) edge[very thick,snakeit] (c);
 \draw (b) edge[crossed,<-] (d);
 \draw (c) edge[very thick,snakeit,bend left=15] (d);
 \draw (c) edge[bend right=15,crossed,<-] (d);
\end{tikzpicture}
$$
\caption{\label{fig:complexSigma}
For a core graph $\Phi$ on the left with chosen edges in the middle, an example of a graph in $\mO\Phi_2$ is drawn on the right.}
\end{figure}
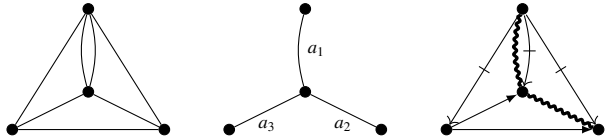

It is clear that
\begin{equation}
\left(\mO\Phi_0,\delta_E\right)=
\left(\mO\Phi,\delta_E\right).
\end{equation}

\begin{lemma}[{\cite[Lemma 4.3]{MultiSourced}}]
\label{lem:QIf}
For every $i=1,\dots,v-1$ there is a quasi-isomorphism $f_i:\mO\Phi_{i-1}\rightarrow\mO\Phi_i$.
\end{lemma}

Recall that the essential difference between $\mO\Phi_{i-1}$ and $\mO\Phi_i$ is in the edge $a_i$, it has to be of type $\EE$ in $\mO\Phi_i$, and it is of another type in $\mO\Phi_{i-1}$. The map $f_i:\mO\Phi_{i-1}\rightarrow\mO\Phi_i$ does not change other edges, and maps $a_i$ as follows.
\begin{equation}
\label{eq:fi}
\Ess\mapsto 0,\quad
\Ed\mapsto\EE,\quad
\dE\mapsto(-1)^{d}\EE.
\end{equation}

Similarly to all complexes, we will define the the involution $\iota$ on $\mO\Phi^i$ for every $i$, such that it commutes with the quasi-isomorphisms $f_i$:
\begin{equation}
\label{eq:lrO}
\Ga  \lonm
-\left\{\Ba{cc} (-1)^{v+\text{number of edges of type $\Ed$ or $\EE$}} \st{\Ga}  & \text{for $d$ even} \\
(-1)^{v+\text{number of edges of the type $\EE$}} \st{\Ga} & \text{for $d$ odd,}
\Ea
\right.
\end{equation}
where $\st{\Ga}$ is gained from $\Ga$ by reversing edges of types $\Ed$ or $\dE$, while retaining other edges.
Similarly as before, $\mO\Phi_i$ splits into the eigenspaces of eigenvectors $+1$ and $-1$ of the involution $\iota$:
 $$
 \mO\Phi_i =\mO\Phi^+_i \oplus \mO\Phi^-_i.
 $$
Lemma \ref{lem:QIf} implies that the restriction $f_i:\mO\Phi_{i-1}^-\rightarrow\mO\Phi_i^-$ is a quasi-isomorphism for every $i=1,\dots,v-1$.
Then the composition
\begin{equation}
f:=f_{v-1}\circ\dots\circ f_1:
{\mO}\Phi^-\rightarrow{\mO}\Phi_{v-1}^{-}
\end{equation}
is also a quasi-isomorphism.

The complex $\mO\Phi_{v-1}$ is one dimensional, spanned by the graph with a spanning tree of the edges of type $\EE$, and all other edges of the type $\Ess$, since every edge of type $\Ed$ would form a cycle. The involution $\iota$ sends that graph to itself, cf.\ \eqref{eq:lrO}, therefore it is in the part $\mO\Phi_{v-1}^+$, so $\mO\Phi_{v-1}^-$ is zero space. Therefore ${\mO}\Phi^-$ is acyclic because of the quasi-isomorphism $f$. This finishes the proof.
\end{proof}

\begin{thm}
\label{thm:main}
\begin{itemize}
\item[]
\item The inclusion $\left(\mO\G_d^+,\p\right)\hookrightarrow\left(\mO\G_d,\p\right)$ is a quasi-isomorphism.
\item Involution $\iota$ on $\left(\mO\G_d,\p\right)$ induces the identity on its homology $H\left(\mO\G_d,\p\right)$.
\end{itemize}
\end{thm}
\begin{proof}
Straightforward from Propositions \ref{prop:skeleton} and \ref{prop:acyclic}.
%
\end{proof}

Let $\OGC_d$ be the graph complex dual to $\OG_d$; its differential splits vertices and commutes with the direction involution given by the same formula (\ref{3: involution on OG_d}) so that $\OGC_d$ splits into a direct sum $\OGC_d^+\oplus \OGC_d^-$ spanned by graphs with eigenvalue $+1$ and $-1$ respectively.

\begin{cor}\label{3: Corollary about OGC^-} { The complex $\OGC_d^-$ is acyclic. In particular, $H^0(\OGC_3^+)=\grt_1$.}
\end{cor}

{\Large
\section{\bf Direction involution on $\HoLBd$ and  Grothendieck-Teichm\"uller group}
}

\subsection{On cohomology of the derivation complex of $\HoLBd$}
Let $\wHoLBd$ be the genus completion of the dg properad $\HoLBd$.
As $\HoLBd$ is free, every continuous derivation $D$ of $\wHoLBd$ is uniquely determined by its values on the generators,
$$
D :\Ba{c}\resizebox{10mm}{!}  {\xy
(0,9)*{^{1\ \ \ 2\ \hspace{7mm}\ \ m}},
(0,-9)*{_{1\ \ \ 2\ \hspace{7mm}\ \ n}},
(0,4.9)*+{...},
(0,-4.9)*+{...},
(0,0)*{\bu}="o",
(-7,7)*{}="1",
(-3,7.0)*{}="2",
(3,7)*{}="3",
(7,7.0)*{}="4",
(-3,-7)*{}="5",
(3,-7)*{}="6",
(7,-7)*{}="7",
(-7,-7)*{}="8",
\ar @{->} "o";"1" <0pt>
\ar @{->} "o";"2" <0pt>
\ar @{->} "o";"3" <0pt>
\ar @{->} "o";"4" <0pt>
\ar @{<-} "o";"5" <0pt>
\ar @{<-} "o";"6" <0pt>
\ar @{<-} "o";"7" <0pt>
\ar @{<-} "o";"8" <0pt>
\endxy}\Ea
\lon \wHoLBd.
$$
Thus the Lie algebra of derivations of $\wHoLBd$ can be identified, as a graded vector space, with the space of coinvariants \cite{MW1}
$$
\Der(\wHoLBd) =
 \prod_{m,n\geq 1} \left(\wHoLBd(m,n) \otimes \sgn_m^{\ot |d|}\otimes \sgn_d^{\ot |d|}\right)^{\bS_m\times \bS_d}[1+2d -d(m+n)]
$$
The differential $\delta$ in $\HoLBd$ makes $\Der(\wHoLBd)$  into a complex with the differential $\Delta=[\delta,\ ]$; this complex is spanned by directed  graphs which have incoming and outgoing  legs (whose labels are (skew)symmetrized and hence omitted in the pictures), for example
$$
\Ga= \resizebox{15mm}{!}{ \xy
(0,0)*{\bu}="d1",
(10,0)*{\bu}="d2",
(-5,-5)*{}="dl",
(5,-5)*{}="dc",
(15,-5)*{}="dr",
(0,10)*{\bu}="u1",
(10,10)*{\bu}="u2",
(5,15)*{}="uc",
(15,15)*{}="ur",
(0,15)*{}="ul",
\ar @{<-} "d1";"d2" <0pt>
\ar @{<-} "d1";"dl" <0pt>
\ar @{<-} "d1";"dc" <0pt>
\ar @{<-} "d2";"dr" <0pt>
\ar @{<-} "u1";"d1" <0pt>
\ar @{<-} "u1";"d2" <0pt>
\ar @{<-} "u2";"d2" <0pt>
\ar @{<-} "u2";"d1" <0pt>
\ar @{<-} "uc";"u2" <0pt>
\ar @{<-} "ur";"u2" <0pt>
\ar @{<-} "ul";"u1" <0pt>
\endxy} \in \Der(\wHoLBd)
$$
These graphs $\Ga$ come equipped with an orientations $or$ given by an ordering of its vertices (up to an even permutation).
 The value of the differential $\Delta$ on an element $\Ga\in \Der(\wHoLBd)$ is obtained by splitting vertices of $\Ga$  and by attaching new $(p,q)$ corollas at each single external (incoming or outgoing) leg for all possible values of $p$ and $q$ of $\Ga$ (cf.\ \cite{MW1}),
$$
\Delta \Gamma =
 \delta\Gamma
 \pm
 \sum\Ba{c}
 \resizebox{9mm}{!}{ \xy
 (0,0)*+{\Ga}="Ga",
(-5,5)*{\bu}="0",
(-8,2)*{}="-1",
(-8,8)*{}="1",
(-5,8)*{}="2",
(-2,8)*{}="3",
(-5,2)*{}="2'",
\ar @{<-} "0";"Ga" <0pt>
\ar @{<-} "0";"-1" <0pt>
\ar @{->} "0";"1" <0pt>
\ar @{->} "0";"2" <0pt>
\ar @{->} "0";"3" <0pt>
\ar @{<-} "0";"2'" <0pt>
 \endxy}\Ea
  \pm
 \sum\Ba{c}
\resizebox{9mm}{!}{  \xy
 (0,0)*+{\Ga}="Ga",
(-5,-5)*{\bu}="0",
(-8,-2)*{}="-1",
(-8,-8)*{}="1",
(-5,-2)*{}="2",
(-2,-8)*{}="3",
\ar @{->} "0";"Ga" <0pt>
\ar @{->} "0";"-1" <0pt>
\ar @{<-} "0";"1" <0pt>
\ar @{->} "0";"2" <0pt>
\ar @{<-} "0";"3" <0pt>
 \endxy}\Ea
 $$
where $\delta$ is the standard differential in $\HoLBd$, i.e.\ it acts on the vertices of $\Ga$ by formula (\ref{3: d in HoLBd_infty}). We define an involution of the dg Lie algebra of derivations as follows (cf. (\ref{2: involution on HoLBd}))
$$
\Ba{rccc}
i: & \Der(\wHoLBd) & \lon & \Der(\wHoLBd)\\
        & \Ga & \lon & (-1)^{(\#V(\Ga)+1)}\st{\Ga}
\Ea
$$
where $\st{\Ga}$ is the graph obtained from $\Ga$ by reversing simultaneously directions on all edges and legs but keeping its orientation (i.e.\ the ordering of vertices) unchanged. This map decomposes the derivation complex into a direct sum,
$$
\Der(\wHoLBd)=\Der(\wHoLBd)^+ \oplus \Der(\wHoLBd)^{-}
$$
corresponding to the eigenvalues $+1$ and $-1$.

\begin{prop}\label{4: prop on acyclicity of Der^-HoLB}
{ The complex $\Der(\wHoLBd)^{-}$ is acyclic while $H^\bu(\Der(\wHoLBd)^+)=\K\oplus H^\bu(\OGC_{2d+1})$}.
\end{prop}

\begin{proof}
There is a canonical morphism of dg Lie algebras constructed in \cite{MW1} (in a slightly more generality),
\Beq\label{4: Morhism F from OGC to}
\Ba{rccc}
 \cF\colon & \OGC_{2d+1} &\to & \Der(\wHoLBd)\\
         &   \Ga & \to & \cF(\Ga)
         \Ea
\Eeq
where the values of the derivation $\cF(\Ga)$
on the generators of the $\HoLBd$ are
given,  by definition, as follows
\Beq \label{equ:def GC action 1}
\cF(\Ga)\left(
\Ba{c}\resizebox{9.7mm}{!}  {\xy
(0,9)*{^{1\ \ \ 2\ \hspace{7mm}\ \ m}},
(0,-9)*{_{1\ \ \ 2\ \hspace{7mm}\ \ n}},
(0,4.9)*+{...},
(0,-4.9)*+{...},
(0,0)*{\bu}="o",
(-7,7)*{}="1",
(-3,7.0)*{}="2",
(3,7)*{}="3",
(7,7.0)*{}="4",
(-3,-7)*{}="5",
(3,-7)*{}="6",
(7,-7)*{}="7",
(-7,-7)*{}="8",
\ar @{->} "o";"1" <0pt>
\ar @{->} "o";"2" <0pt>
\ar @{->} "o";"3" <0pt>
\ar @{->} "o";"4" <0pt>
\ar @{<-} "o";"5" <0pt>
\ar @{<-} "o";"6" <0pt>
\ar @{<-} "o";"7" <0pt>
\ar @{<-} "o";"8" <0pt>
\endxy}\Ea\right)
:=
 \sum_{[n]\rar V(\Ga)\atop [m]\rar V(\Ga)}  \Ba{c}\resizebox{11.5mm}{!}  {\xy
 (-6,7)*{^1},
(-3,7)*{^2},
(2.5,7)*{},
(7,7)*{^m},
(-3,-8)*{_2},
(3,-6)*{},
(7,-8)*{_n},
(-6,-8)*{_1},
(0,4.5)*+{...},
(0,-4.5)*+{...},
(0,0)*+{\Ga}="o",
(-6,6)*{}="1",
(-3,6)*{}="2",
(3,6)*{}="3",
(6,6)*{}="4",
(-3,-6)*{}="5",
(3,-6)*{}="6",
(6,-6)*{}="7",
(-6,-6)*{}="8",
\ar @{->} "o";"1" <0pt>
\ar @{->} "o";"2" <0pt>
\ar @{->} "o";"3" <0pt>
\ar @{->} "o";"4" <0pt>
\ar @{<-} "o";"5" <0pt>
\ar @{<-} "o";"6" <0pt>
\ar @{<-} "o";"7" <0pt>
\ar @{<-} "o";"8" <0pt>
\endxy}\Ea
\Eeq
 where the sum is taken over all ways of attaching the incoming $n$ legs and outgoing $m$ legs to the graph $\Ga$, and setting to zero every resulting graph if it contains a vertex with valency $<3$ or
   with no at least one incoming  or at least one outgoing edge. This map is proven in \cite{MW1} to be a quasi-isomorphism up to one rescaling class (of degree zero) given by the following element of $\Der(\wHoLBd)$
$$
r=  \sum_{m,n}(m+n-2)
  \overbrace{
  \underbrace{
 \Ba{c}\resizebox{8mm}{!}  {\xy
(0,4.5)*+{...},
(0,-4.5)*+{...},
(0,0)*{\bu}="o",
(-6,5)*{}="1",
(-3,5)*{}="2",
(3,5)*{}="3",
(6,5)*{}="4",
(-3,-5)*{}="5",
(3,-5)*{}="6",
(6,-5)*{}="7",
(-6,-5)*{}="8",
\ar @{->} "o";"1" <0pt>
\ar @{->} "o";"2" <0pt>
\ar @{->} "o";"3" <0pt>
\ar @{->} "o";"4" <0pt>
\ar @{<-} "o";"5" <0pt>
\ar @{<-} "o";"6" <0pt>
\ar @{<-} "o";"7" <0pt>
\ar @{<-} "o";"8" <0pt>
\endxy}\Ea
 }_n
 }^{m}.
 $$
The maps $\cF$ respects orientations of graphs and commutes with the direction involution of both sides, and hence induces a strict quasi-isomorphism
$$
\cF^-: \OGC_{2d+1}^- \lon \Der(\wHoLBd)^-
$$
and a quasi-isomorphism up to one rescaling class
$$
\cF^+: \OGC_{2d+1}^+ \lon \Der(\wHoLBd)^+
$$
implying with the help of Corollary {\ref{3: Corollary about OGC^-}} the required claims.
\end{proof}

\subsection{A suitable extension of $\OGC_d$} The direct sum of complexes
$$
\OGC_d^{ext}:= \K\oplus \OGC_{d}
$$
can be made into a dg Lie algebra by extending the  standard  Lie bracket in $\OGC_d$ to the basis vector $1\in \K$ (``the unique graph with no vertices and edges") as follows \cite{MW2}
$$
[1,1]=0, \ \  [1,\Gamma]:=2(\#V(\Ga) - \#E(\Ga))\Ga.
$$
Then $H^0(\OGC_d^{ext})=\grt$, the Lie algebra of the full Grothendieck-Techm\"uller group $GRT$ (rather than of its subgroup $GRT_1$), and the above map (\ref{4: Morhism F from OGC to}) extends to a quasi-isomorphism of dg Lie algebras
\Beq\label{4: F^ext}
\cF^{ext}\colon  \OGC_{2d+1}^{ext} \to  \Der(\wHoLBd)
\Eeq
which send the extra generator $1$ into the rescaling class $r$. The involution $\imath$ extends to $\OGC_d^{ext}$ in the obvious way,
$$
\OGC_d^{ext+}= \OGC_d^{+} \oplus \K, \ \ \ \OGC_d^{ext-}=\OGC_d^{-},
$$
keeping the eigenvalue $-1$ subspace acyclic.

\subsection{On the action of the Grothendieck-Teichm\"uller group\ on $\LB$} Since the natural epimorphism
$$
p: \HoLBd\lon \LBd,
$$
is a quasi-isomorphism,
the derivation complex $\Der(\wHoLBd)$ admits a much smaller incarnation \cite{MW1}
$$
\Der(\wHoLBd)[-1]\simeq \Def(\wHoLBd \stackrel{p}{\rar} \widehat{\LB}_d) \equiv
 \prod_{m,n\geq 1} \left(\widehat{\LB}_d(m,n) \otimes \sgn_m^{\ot |d|}\otimes \sgn_n^{\ot |d|}\right)^{\bS_m\times \bS_m}[d(2-m-n)]
$$
as a deformation complex of the above map $p$ given in terms of solely trivalent graphs. The action of $GRT$ on $\widehat{\LB}$ is the exponentiation of the action of $\grt$ on $\widehat{\LB}$ by derivations  which in turn depends (up to homotopy equivalence) on the representation of any element of $[g]\in \grt$ as a cycle in the graph complex $\OGC_3$. By
Corollary {\ref{3: Corollary about OGC^-} and Proposition \ref{4: prop on acyclicity of Der^-HoLB}, one can always find a cycle representative $g$ of $[g]$ which belongs to $\OGC_3^{ext+}$ and hence satisfies a condition that its actions on the generators of $\LB$ (which fully determines the action of $g$ on the whole properad $\widehat{\LB}$)
commutes with the involution, i.e.
$$
\cF(g)\left(
\Ba{c}\resizebox{6mm}{!}{  \xy
(0,-7)*{_{1\hspace{5mm} 2}},
(0,5)*{}="U1";
(0,0)*{\circ}="o";
(-3,-5)*{}="D1";
(3,-5)*{}="D2";
\ar @{->} "D1";"o" <0pt>
\ar @{->} "D2";"o" <0pt>
\ar @{<-} "U1";"o" <0pt>
 \endxy}
 \Ea\right)=\imath \circ \cF(g)\left(\Ba{c}\resizebox{6mm}{!}{  \xy
(0,7)*{_{1\hspace{5mm} 2}},
(0,-5)*{}="U1";
(0,0)*{\circ}="o";
(-3,5)*{}="D1";
(3,5)*{}="D2";
\ar @{<-} "D1";"o" <0pt>
\ar @{<-} "D2";"o" <0pt>
\ar @{->} "U1";"o" <0pt>
 \endxy}
 \Ea \right)\ \  \text{or, equivalently}, \ \
  \cF(g)\left(\Ba{c}\resizebox{6mm}{!}{  \xy
(0,7)*{_{1\hspace{5mm} 2}},
(0,-5)*{}="U1";
(0,0)*{\circ}="o";
(-3,5)*{}="D1";
(3,5)*{}="D2";
\ar @{<-} "D1";"o" <0pt>
\ar @{<-} "D2";"o" <0pt>
\ar @{->} "U1";"o" <0pt>
 \endxy}
 \Ea \right) =\imath \circ
 \cF(g)\left(
\Ba{c}\resizebox{6mm}{!}{  \xy
(0,-7)*{_{1\hspace{5mm} 2}},
(0,5)*{}="U1";
(0,0)*{\circ}="o";
(-3,-5)*{}="D1";
(3,-5)*{}="D2";
\ar @{->} "D1";"o" <0pt>
\ar @{->} "D2";"o" <0pt>
\ar @{<-} "U1";"o" <0pt>
 \endxy}
 \Ea\right),
$$
This proves the positive answer to {\sf Question 1} from the Introduction.

\bip

{\Large
\section{\bf An application to the deformation quantization theory}
}

\subsection{Universal quantizations of Lie bialgebras as morphisms of props}

There is a polydifferential endofunctor \cite{MW2}
$$
\caD: \text{\sf Category of dg props} \lon \text{\sf Category of dg props}
$$
which  has the property
that for any dg prop $\cP$ and its any representation, $\rho: \cP\rar \cE nd_V$,
 in a dg vector space $V$ the associated dg prop $\caD\cP$ has an associated representation, $\caD\rho: \caD\cP\rar \cE nd_{{\odot^\bu} V}$,
 in the graded commutative algebra ${\odot^\bu} V$ given in terms of polydifferential (with respect to the standard multiplication in ${\odot^\bu} V$) operators. We refer to \S 5.2 of \cite{MW2} for full details and explain briefly only the explicit structure of the dg prop\footnote{Any dg properad $\cP$ has an associated dg prop which is denoted often by the same letter (because the associated functor from the category of properads to the category of props is exact).}
 $
\caD{\LB}$. If elements of $\LB$ has all legs labelled differently, the elements of
 $\caD{\LB}$ can be understood as graphs from $\LB$ whose different in-legs (or out-legs) may have  identical numerical labels. For example, with the element
 $$
 \Ba{c}\resizebox{9mm}{!}{
\xy
(+3,11)*{\bu}="0",
 (-3,8)*{\bu}="a",
(-5,2)*+{_1}="b_1",
(5,2)*+{_2}="b_2",
(-5,18)*+{^1}="u_1",
(5,18)*+{^2}="u_2",
\ar @{->} "a";"0" <0pt>
\ar @{<-} "a";"b_1" <0pt>
\ar @{<-} "0";"b_2" <0pt>
\ar @{->} "a";"u_1" <0pt>
\ar @{->} "0";"u_2" <0pt>
\endxy}
\Ea
  \in \LB(2,2)
$$
can generate several different elements in $\caD{\LB}$ depending on which legs receive identical labels, e.g.
 $$
  \Ba{c}\resizebox{9mm}{!}{
\xy
(+3,11)*{\bu}="0",
 (-3,8)*{\bu}="a",
(-5,2)*+{_1}="b_1",
(5,2)*+{_1}="b_2",
(-5,18)*+{^1}="u_1",
(5,18)*+{^2}="u_2",
\ar @{->} "a";"0" <0pt>
\ar @{<-} "a";"b_1" <0pt>
\ar @{<-} "0";"b_2" <0pt>
\ar @{->} "a";"u_1" <0pt>
\ar @{->} "0";"u_2" <0pt>
\endxy}
\Ea
  \in \caD\LB(2,1), \ \ \ \
   \Ba{c}\resizebox{9mm}{!}{
\xy
(+3,11)*{\bu}="0",
 (-3,8)*{\bu}="a",
(-5,2)*+{_1}="b_1",
(5,2)*+{_2}="b_2",
(-5,18)*+{^1}="u_1",
(5,18)*+{^1}="u_2",
\ar @{->} "a";"0" <0pt>
\ar @{<-} "a";"b_1" <0pt>
\ar @{<-} "0";"b_2" <0pt>
\ar @{->} "a";"u_1" <0pt>
\ar @{->} "0";"u_2" <0pt>
\endxy}
\Ea
  \in \caD\LB(1,2).
$$
Hence the elements of $\caD{\LB}$ can be best described by adding to graph representations of elements of $\widehat{\LB}$  new white labelled in-vertices (as well as labelled out-vertices) which can have several, or none, in-legs (resp., out-legs) attached to them, for example
 $$
\Ba{c}\resizebox{9mm}{!}{
\xy
(+3,11)*{\bu}="0",
 (-3,8)*{\bu}="a",
(0,2)*+{_1}*\frm{o}="b";
(-5,18)*+{_1}*\frm{o}="u_1";
(5,18)*+{_2}*\frm{o}="u_2";
\ar @{->} "a";"0" <0pt>
\ar @{<-} "a";"b" <0pt>
\ar @{<-} "0";"b" <0pt>
\ar @{->} "a";"u_1" <0pt>
\ar @{->} "0";"u_2" <0pt>
\endxy}
\Ea\in \caD\LB(2,1), \ \ \ \
\Ba{c}\resizebox{9mm}{!}{
\xy
(+3,11)*{\bu}="0",
 (-3,8)*{\bu}="a",
(0,18)*+{_1}*\frm{o}="u";
(-5,2)*+{_1}*\frm{o}="b_1";
(5,2)*+{_2}*\frm{o}="b_2";
\ar @{->} "a";"0" <0pt>
\ar @{<-} "a";"b_1" <0pt>
\ar @{<-} "0";"b_2" <0pt>
\ar @{->} "a";"u" <0pt>
\ar @{->} "0";"u" <0pt>
\endxy}
\Ea\in \caD\LB(1,2).
$$
The functor $\caD$ constructs props acting on symmetric tensor algebras  $\odot^\bu V$ which have canonical multiplication and co-multiplication.  These two canonical operations gets encoded into the structure of $\caD\LB$ by allowing graphs as above with no legs attached to the new labelled white vertices; for example graphs of the form
$\Ba{c}\resizebox{6mm}{!}{
\xy
(0,3)*+{_1}*\frm{o};
(-3,-3)*+{_1}*\frm{o};
(3,-3)*+{_2}*\frm{o};
\endxy}\Ea$ and $\Ba{c}\resizebox{6mm}{!}{
\xy
(0,-3)*+{_1}*\frm{o};
(-3,3)*+{_1}*\frm{o};
(3,3)*+{_2}*\frm{o};
\endxy}\Ea
 $ are allowed in $\caD(\widehat{\LB})$,
 and control precisely the aforementioned canonical operations in $\odot^\bu V$.

\sip

It has been shown in \cite{MW2} that any universal  quantization of (possibly, infinite-dimensional) Lie bialgebras can be understood as a morphism of props
\Beq\label{5: morphism F}
\cQ: \Assb \lon \caD\widehat{\LB}
\Eeq
satisfying a certain non-triviality condition. Here $\Assb$ is the prop of associative bialgebras which has, by definition,  two sets of generators controlling compatible multiplication and co-multiplication which are associative but not necessarily (co)commutative.  More precisely, $\cA ss\cB$ is the quotient,
$$
\Assb:= {\cF ree\langle A \rangle}/(R)
$$
of the free prop, $\cF ree\langle A \rangle$, generated by an $\bS$-bimodule $A=\{A(m,n)\}$ with all $A(m,n)$ zero except the following ones (cf.\ \S 2)
$$
A(2,1):=
\K[\bS_2]\ot \id_1\equiv\mbox{span}\left\langle \hspace{-2mm}
\Ba{c}\resizebox{6mm}{!}{  \xy
(0,7)*{_{1\hspace{5mm} 2}},
(0,-5)*{}="U1";
(0,0)*{\circ}="o";
(-3,5)*{}="D1";
(3,5)*{}="D2";
\ar @{<-} "D1";"o" <0pt>
\ar @{<-} "D2";"o" <0pt>
\ar @{->} "U1";"o" <0pt>
 \endxy}
 \Ea
\hspace{-1mm},\hspace{-1mm}
\Ba{c}\resizebox{6mm}{!}{  \xy
(0,7)*{_{2\hspace{5mm} 1}},
(0,-5)*{}="U1";
(0,0)*{\circ}="o";
(-3,5)*{}="D1";
(3,5)*{}="D2";
\ar @{<-} "D1";"o" <0pt>
\ar @{<-} "D2";"o" <0pt>
\ar @{->} "U1";"o" <0pt>
 \endxy}
 \Ea
 \hspace{-1mm}  \right\rangle
 ,\ \ \ \ A(1,2):=
\id_1\ot \K[\bS_2]\equiv
\mbox{span}\left\langle
 \hspace{-2mm}
\Ba{c}\resizebox{6mm}{!}{  \xy
(0,-7)*{_{1\hspace{5mm} 2}},
(0,5)*{}="U1";
(0,0)*{\circ}="o";
(-3,-5)*{}="D1";
(3,-5)*{}="D2";
\ar @{->} "D1";"o" <0pt>
\ar @{->} "D2";"o" <0pt>
\ar @{<-} "U1";"o" <0pt>
 \endxy}
 \Ea
 \hspace{-1mm} ,  \hspace{-1mm}
\Ba{c}\resizebox{6mm}{!}{  \xy
(0,-7)*{_{2\hspace{5mm} 1}},
(0,5)*{}="U1";
(0,0)*{\circ}="o";
(-3,-5)*{}="D1";
(3,-5)*{}="D2";
\ar @{->} "D1";"o" <0pt>
\ar @{->} "D2";"o" <0pt>
\ar @{<-} "U1";"o" <0pt>
 \endxy}
 \Ea
  \hspace{-1mm}
\right\rangle
$$
modulo the ideal generated by relations
$$
R:\left\{
\Ba{c}\resizebox{8mm}{!}{  \xy
(0,7)*{_{\ \hspace{5mm} 3}},
(-3,12)*{_{1 \hspace{5mm} 2}},
(0,-5)*{}="U1";
(0,0)*{\circ}="o";
(-3,5)*{\circ}="D1";
(3,5)*{}="D2";
(-6,10)*{}="DD1";
(0,10)*{}="DD2";
\ar @{<-} "D1";"o" <0pt>
\ar @{<-} "D2";"o" <0pt>
\ar @{->} "U1";"o" <0pt>
\ar @{->} "D1";"DD1" <0pt>
\ar @{->} "D1";"DD2" <0pt>
 \endxy}
 \Ea
\ - \
\Ba{c}\resizebox{8mm}{!}{  \xy
(0,7)*{_{1 \hspace{5mm} \ }},
(3,12)*{_{2 \hspace{5mm} 3}},
(0,-5)*{}="U1";
(0,0)*{\circ}="o";
(3,5)*{\circ}="D1";
(-3,5)*{}="D2";
(6,10)*{}="DD1";
(0,10)*{}="DD2";
\ar @{<-} "D1";"o" <0pt>
\ar @{<-} "D2";"o" <0pt>
\ar @{->} "U1";"o" <0pt>
\ar @{->} "D1";"DD1" <0pt>
\ar @{->} "D1";"DD2" <0pt>
 \endxy}
 \Ea
=0, \ \ \ \ \
%
\Ba{c}\resizebox{8mm}{!}{  \xy
(0,-7)*{_{\ \hspace{5mm} 3}},
(-3,-12)*{_{1 \hspace{5mm} 2}},
(0,5)*{}="U1";
(0,0)*{\circ}="o";
(-3,-5)*{\circ}="D1";
(3,-5)*{}="D2";
(-6,-10)*{}="DD1";
(0,-10)*{}="DD2";
\ar @{->} "D1";"o" <0pt>
\ar @{->} "D2";"o" <0pt>
\ar @{<-} "U1";"o" <0pt>
\ar @{<-} "D1";"DD1" <0pt>
\ar @{<-} "D1";"DD2" <0pt>
 \endxy}
 \Ea
\ - \
\Ba{c}\resizebox{8mm}{!}{  \xy
(0,-7)*{_{1 \hspace{5mm} \ }},
(3,-12)*{_{2 \hspace{5mm} 3}},
(0,5)*{}="U1";
(0,0)*{\circ}="o";
(3,-5)*{\circ}="D1";
(-3,-5)*{}="D2";
(6,-10)*{}="DD1";
(0,-10)*{}="DD2";
\ar @{->} "D1";"o" <0pt>
\ar @{->} "D2";"o" <0pt>
\ar @{<-} "U1";"o" <0pt>
\ar @{<-} "D1";"DD1" <0pt>
\ar @{<-} "D1";"DD2" <0pt>
 \endxy}
 \Ea
=0,\ \ \ \ \ \
%
\Ba{c}\resizebox{6mm}{!}{  \xy
(0,9.5)*{_{1\hspace{5mm} 2}},
%
(0,2.5)*{\circ}="0U";
(-3,7.5)*{}="U1";
(3,7.5)*{}="U2";
(0,-9.5)*{_{1\hspace{5mm} 2}},
%
(0,-2.5)*{\circ}="0D";
(-3,-7.5)*{}="D1";
(3,-7.5)*{}="D2";
\ar @{<-} "U1";"0U" <0pt>
\ar @{<-} "U2";"0U" <0pt>
\ar @{->} "0D";"0U" <0pt>
\ar @{->} "D1";"0D" <0pt>
\ar @{->} "D2";"0D" <0pt>
 \endxy}
 \Ea
\ - \
\Ba{c}\resizebox{13mm}{!}{  \xy
(0,17)*{^{1\hspace{9mm} 2}},
(0,-7)*{_{1\hspace{9mm} 2}},
(-5,-5)*{}="U1'";
(-5,0)*{\circ}="o'";
(-10,5)*{}="D1'";
(-2,5)*{}="D2'";
(5,-5)*{}="U1''";
(5,0)*{\circ}="o''";
(10,5)*{}="D1''";
(2,5)*{}="D2''";
(-5,10)*{\circ}="p'";
(5,10)*{\circ}="p''";
(-5,15)*{}="u'";
(5,15)*{}="u''";
\ar @{<-} "p''";"o'" <0pt>
\ar @{<-} "D1'";"o'" <0pt>
\ar @{->} "U1'";"o'" <0pt>
\ar @{<-} "p'";"o''" <0pt>
\ar @{<-} "D1''";"o''" <0pt>
\ar @{->} "U1''";"o''" <0pt>
\ar @{->} "D1''";"p''" <0pt>
\ar @{->} "D1'";"p'" <0pt>
\ar @{<-} "u''";"p''" <0pt>
\ar @{<-} "u'";"p'" <0pt>
 \endxy}
 \Ea
=0
\right.
$$
As we see, the relations are invariant under a reversing direction involution,
$$
\imath: \Ba{c}\resizebox{6mm}{!}{  \xy
(0,-7)*{_{1\hspace{5mm} 2}},
(0,5)*{}="U1";
(0,0)*{\circ}="o";
(-3,-5)*{}="D1";
(3,-5)*{}="D2";
\ar @{->} "D1";"o" <0pt>
\ar @{->} "D2";"o" <0pt>
\ar @{<-} "U1";"o" <0pt>
 \endxy}
 \Ea \lon \Ba{c}\resizebox{6mm}{!}{  \xy
(0,7)*{_{1\hspace{5mm} 2}},
(0,-5)*{}="U1";
(0,0)*{\circ}="o";
(-3,5)*{}="D1";
(3,5)*{}="D2";
\ar @{<-} "D1";"o" <0pt>
\ar @{<-} "D2";"o" <0pt>
\ar @{->} "U1";"o" <0pt>
 \endxy}
 \Ea, \
 \Ba{c}\resizebox{6mm}{!}{  \xy
(0,7)*{_{1\hspace{5mm} 2}},
(0,-5)*{}="U1";
(0,0)*{\circ}="o";
(-3,5)*{}="D1";
(3,5)*{}="D2";
\ar @{<-} "D1";"o" <0pt>
\ar @{<-} "D2";"o" <0pt>
\ar @{->} "U1";"o" <0pt>
 \endxy}
 \Ea
 \lon
 \Ba{c}\resizebox{6mm}{!}{  \xy
(0,-7)*{_{1\hspace{5mm} 2}},
(0,5)*{}="U1";
(0,0)*{\circ}="o";
(-3,-5)*{}="D1";
(3,-5)*{}="D2";
\ar @{->} "D1";"o" <0pt>
\ar @{->} "D2";"o" <0pt>
\ar @{<-} "U1";"o" <0pt>
 \endxy}
 \Ea
$$
in a full analogy to the prop $\LB$.
Moreover it was proven in op.cit. that the cohomology of the deformation complex \cite{MV} of any such a quantization morphism $\cQ$ can be identified with
the the cohomology of the symmetric tensor algebra of the oriented graph complex (up to one class controlling the rescaling automorphism of $\LB$, cf.\ \S 4),
$$
H^{\bu+1}\left(\Def(\Assb \stackrel{\cQ}{\rar} \caD\widehat{\LB})\right)=\odot^{\geq 1} \left(H^\bu(\OGC_3)\oplus \K\right)
$$
implying the isomorphism
$$
H^{1}\left(\Def(\Assb \stackrel{\cQ}{\rar} \caD\widehat{\LB})\right)= \grt,
$$
where $\grt$ is the Lie algebra of the full Grothendieck-Teichm\"uller group $GRT$
(not just its Lie subalgebra $\grt_1$). As every infinitesimal deformation of
any universal quantization map $\cQ$ can always be exponentiated to a genuine deformation of $\cQ$, this leads us to the identification of the homotopy classes of universal quantization with the set of Drinfeld associators.

\sip

Any universal quantization (\ref{1: involution of Lie and coLie}) is uniquely determined by its values
in the generators of $\Assb$. We call such a quantization morphism {\em direction involution invariant}\, if
$$
\cQ\left(
\Ba{c}\resizebox{6mm}{!}{  \xy
(0,-7)*{_{1\hspace{5mm} 2}},
(0,5)*{}="U1";
(0,0)*{\circ}="o";
(-3,-5)*{}="D1";
(3,-5)*{}="D2";
\ar @{->} "D1";"o" <0pt>
\ar @{->} "D2";"o" <0pt>
\ar @{<-} "U1";"o" <0pt>
 \endxy}
 \Ea\right)=\imath \circ \cQ\left(\Ba{c}\resizebox{6mm}{!}{  \xy
(0,7)*{_{1\hspace{5mm} 2}},
(0,-5)*{}="U1";
(0,0)*{\circ}="o";
(-3,5)*{}="D1";
(3,5)*{}="D2";
\ar @{<-} "D1";"o" <0pt>
\ar @{<-} "D2";"o" <0pt>
\ar @{->} "U1";"o" <0pt>
 \endxy}
 \Ea \right)\ \  \text{or, equivalently}, \ \
  \cQ\left(\Ba{c}\resizebox{6mm}{!}{  \xy
(0,7)*{_{1\hspace{5mm} 2}},
(0,-5)*{}="U1";
(0,0)*{\circ}="o";
(-3,5)*{}="D1";
(3,5)*{}="D2";
\ar @{<-} "D1";"o" <0pt>
\ar @{<-} "D2";"o" <0pt>
\ar @{->} "U1";"o" <0pt>
 \endxy}
 \Ea \right) =\imath \circ
 \cQ\left(
\Ba{c}\resizebox{6mm}{!}{  \xy
(0,-7)*{_{1\hspace{5mm} 2}},
(0,5)*{}="U1";
(0,0)*{\circ}="o";
(-3,-5)*{}="D1";
(3,-5)*{}="D2";
\ar @{->} "D1";"o" <0pt>
\ar @{->} "D2";"o" <0pt>
\ar @{<-} "U1";"o" <0pt>
 \endxy}
 \Ea\right),
$$
where $\imath$ the direction involution of $\LB$ discussed in \S 2.

\begin{thm}
{ The set of homotopy classes of direction involution invariant can be identified with the set of Drinfeld associators.}
\end{thm}

\begin{proof}
 It was proven in \cite{KT} that the set of direction involution invariant universal quantizations is non-empty: there exists at least one such quantization (which the authors of\cite{KT} even suspected to be unique); let us denote it by $\cQ_0$ and study its deformation complex
$
\Def(\Assb \stackrel{\cQ_0}{\lon} \caD\widehat{\LB})
$ in the category of dg props.
There is an obvious morphism of deformation complexes
$$
\Def(\LB \stackrel{\Id}{\rar} {\LB}) \lon \Def(\Assb \stackrel{\cQ_0}{\rar} \caD\widehat{\LB})
$$
which is proven in \cite{MW2} to be a quasi-isomorphism for any universal quantization $\cQ$, in particular for $\cQ_0$. There is also a quasi-isomorphism of complexes
$$
\odot^{\geq 1} \left(\OGC_3\oplus \K\right) \lon \Def(\LB \stackrel{\Id}{\rar} {\LB})[1]
$$
constructed in \cite{MW1}. Composing these two we get a quasi-isomorphism of complexes.
$$
s: \odot^{\geq 1} \left(\OGC_3\oplus \K\right) \lon \Def(\Assb \stackrel{\cQ_0}{\rar} \caD\widehat{\LB})[1]
$$
As $\cQ_0$ is direction involution invariant, the differential in the latter complex commutes with the involution $\imath$ and hence decomposes into a direct sum of complexes,
$$
\Def(\Assb \stackrel{\cQ_0}{\rar} \caD\widehat{\LB})=\Def(\Assb \stackrel{\cF_0}{\rar} \caD\widehat{\LB})^+ [1] \oplus \Def(\Assb \stackrel{\cQ_0}{\rar} \caD\widehat{\LB})^-[1].
$$
The map $s$ respects the direction involution of both sides and hence induces quasi-isomorphisms
$$
s^+: \odot^{\geq 1} \left(\OGC_3\oplus \K\right)^+ \lon \Def(\Assb \stackrel{\cQ_0}{\rar} \caD\widehat{\LB})^+, \ \ \ \
s^-: \odot^{\geq 1} \left(\OGC_3\oplus \K\right)^-\lon \Def(\Assb \stackrel{\cQ_0}{\rar} \caD\widehat{\LB})^-.
$$
By Corollary {\ref{3: Corollary about OGC^-}} the complex $\OGC_3^-$ is acyclic implying
$$
H^1\left(\Def(\Assb \stackrel{\cQ_0}{\rar} \caD\widehat{\LB})^+\right)=\grt, \ \ \ H^\bu\left(\Def(\Assb \stackrel{\cQ_0}{\rar} \caD\widehat{\LB})^-\right)=0.
$$
Since every infinitesimal deformation corresponding to a class from $H^1\left(\Def(\Assb \stackrel{\cQ_0}{\rar} \caD\widehat{\LB})^+\right)$ exponentiates to a genuine involution invariant morphism
\cite{MW2}, and since the set of Drinfeld associators is the principal homogeneous space for $GRT$, the claim of the Theorem follows.
\end{proof}

\begin{cor}
{ Any universal quantization $\cQ$ of Lie bialgebras as in (\ref{5: morphism F}) is homotopy equivalent
to a direction involution invariant one.}
\end{cor}

\end{document}